\documentclass[11pt,article]{article}

\usepackage{amsmath}
\usepackage{amsthm}
  \usepackage{paralist}
  \usepackage{graphics} %% add this and next lines if pictures should be in esp format
  \usepackage{epsfig} %For pictures: screened artwork should be set up with an 85 or 100 line screen
\usepackage{graphicx}
\usepackage{caption}
\usepackage{subcaption}
\usepackage{epstopdf}%This is to transfer .eps figure to .pdf figure; please compile your paper using PDFLeTex or PDFTeXify.
 \usepackage[colorlinks=true]{hyperref}
 \usepackage{multirow}
\input{amssym.tex}

\newtheorem{theorem}{Theorem}[section]

\newtheorem{lemma}[theorem]{Lemma}
\newtheorem{proposition}{Proposition}

 \numberwithin{equation}{section}
\newtheorem{remark}{Remark}

\newcommand{\keywords}

\def\bc{\begin{center}}       \def\ec{\end{center}}
\def\ba{\begin{array}}        \def\ea{\end{array}}
\def\be{\begin{equation}}     \def\ee{\end{equation}}
\def\bea{\begin{eqnarray}}    \def\eea{\end{eqnarray}}
\def\beaa{\begin{eqnarray*}}  \def\eeaa{\end{eqnarray*}}

\def\mathbb{\Bbb}

\begin{document}

\title{\LARGE \bf Qualitative analysis of stationary Keller--Segel chemotaxis models with logistic growth}
\author{Qi Wang \thanks{Department of Mathematics, Southwestern University of Finance and Economics, Chengdu, Sichuan 611130, China
        ({\tt qwang@swufe.edu.cn}).  QW is partially supported by NSF-China (Grant No. 11501460) and the Project (No.15ZA0382) from Department of Education, Sichuan China.}, Jingda Yan \thanks{Hanqing Advanced Institute of Economics and Finance, Renmin University of China, No. 59 Zhongguancun Street, Haidian District, Beijing 100872, China ({\tt jdyan@ruc.edu.cn}).}, Chunyi Gai \thanks{Department of Mathematics and Statistics, Dalhousie University, 6316 Coburg Road, Halifax, Nova Scotia, B3H 4R2, Canada ({\tt chunyi.gai@dal.ca}).}
}
\date{}
\maketitle

\abstract
We study the stationary Keller--Segel chemotaxis models with logistic cellular growth over a one-dimensional region subject to the Neumann boundary condition.  We show that nonconstant solutions emerge in the sense of Turing's instability as the chemotaxis rate $\chi$ surpasses a threshold number.  By taking the chemotaxis rate as the bifurcation parameter, we carry out bifurcation analysis on the system to obtain the explicit formulas of bifurcation values and small amplitude nonconstant positive solutions.  Moreover we show that solutions stay strictly positive in the continuum of each branch.  The stabilities of these steady state solutions are well studied when the creation and degradation rate of the chemical is assumed to be a linear function.  Finally we investigate the asymptotic behaviors of the monotone steady states.  We construct solutions with interesting patterns such as a boundary spike when the chemotaxis rate is large enough and/or the cell motility is small.

\textbf{Keywords: Chemotaxis, logistic growth, steady state, bifurcation analysis, asymptotic behavior}\\
\textbf{AMS Class: 92C17, 35B20, 35B32, 35B40}

\section{Introduction and preliminary results}\label{section1}
\indent
In this paper, we investigate stationary Keller--Segel type chemotaxis models with cellular growth in the following form
\begin{equation}\label{11}
\left\{
\begin{array}{ll}
(D_1 u' - \chi \Phi(u,v)v')'+(\bar{u}-u)u=0,  & x \in (0,L),\\
D_2 v''-v+h(u)=0, & x \in (0,L),\\
u'(x)=v'(x)=0,&x=0,L,\\
\end{array}
\right.
\end{equation}
where $D_1$, $D_2$, $\chi$ and $\bar u$ are positive constants.  $u,v$ are functions of $x$ and $\Phi(y,v)$, $h(u)$ are smooth functions.

Chemotaxis is the directed movement of microorganisms along the gradient of certain chemical stimulus which may be produced by the cells and consumed by certain enzyme.  It has attracted significant interest from numerous scientists over the past few decades due to its important applications in a wide range of biological phenomena, such as wound healing, embryonic development and cancer growth of tumour cells \cite{BWS,CP,DW,N}.

Mathematical modelling and theoretical analysis of chemotaxis date to the pioneering works of Keller and Segel \cite{KS,KS1,KS2} in the early 1970s.  In its original form, the Keller--Segel type chemotaxis model consists of four reaction-advection-diffusion equations which can be reduced into two coupled nonlinear PDEs.  One is a convection-diffusion equation for the cell population density and the other one is a reaction-diffusion equation for the chemical concentration.  Let $\Omega $ be a bounded domain in $\mathbb{R}^N$, $N \geq 1$ and denote by $u(x,t)$ the cell population density and by $v(x,t)$ the chemical concentration at location $x$ and time $t$ respectively.  Then the general form of a classical Keller--Segel chemotaxis system reads
\begin{equation}\label{12}
\left\{
\begin{array}{ll}
u_t=\nabla\cdot(D_1(u,v) \nabla u - \chi \Phi(u,v) \nabla v),  & x \in\Omega,~t>0,\\
v_t=D_2 (u,v) \Delta v + g(u,v), & x \in\Omega,~t>0,\\
\end{array}
\right.
\end{equation}
where $D_1>0$ is the so-called cell \emph{motility} and it interprets the tendency of cells to move randomly over the domain $\Omega$, and $D_2>0$ is the diffusion rate of the chemical.  The constant $\chi$ measures the strength of influence of the chemical on the directed cellular movement; moreover $\chi>0$ if the chemical is attractive to the cells and $\chi<0$ if the chemical is repulsive.  We will focus on the former case in this paper and assume that $\chi>0$ from now on.  $\Phi(u,v)$ is called the \emph{sensitivity function} and it reflects the variation of cellular sensitivity to the stimulus with respect to the population density of cells and levels of chemical concentration.  $g(u,v)$ is the creation and degradation rate of the chemical.

One of the most important and interesting phenomena of chemotaxis is the cellular aggregation, in which initially evenly distributed cells merge with each other and eventually aggregate into one or several groups.  The variation of the functions in (\ref{12}) allows this Keller--Segel model to admit very rich spatial--temporal dynamics from the view point of mathematical analysis.  It can induce various interesting and striking properties such as global existence, finite time blow-ups etc. to model the cellular aggregation.  Moreover, this intuitively simple system successfully demonstrates its ability in presenting solutions with spatial patterns even in its simplest case.  Furthermore, (\ref{12}) is able to explain the phenomenon of wave propagation of bands of certain bacteria under the influence of a chemical.

Time-dependent system in the form of (\ref{12}) can describe cell aggregation when the solutions blow up with their $L^\infty$ norms going to infinity within finite or infinite time.  Then the aggregation is simulated by a single $\delta$-function or a linear combination of several $\delta$-functions \cite{CP,HV,HW,N}.  An alternative way proposed is to show that the time-dependent system (\ref{12}) admits global-in-time solutions which converge to bounded steady states.  Moreover, these steady states can create interesting patterns such as spikes, or transition layers, which can be used to model the cell aggregations.  For $N\geq1$, Ni and Takagi \cite{NT,NT2} converted a chemotaxis steady state system into a single equation, of which they obtained nonconstant positive solutions by variational method.  Moreover, they constructed a steady state solution that concentrates on the most curved part of the boundary as the chemical diffusion rate shrinks to zero.  See the survey paper \cite{N2} for works and recent developments in this direction.  For $N=1$, Wang \cite{W} initiated a method, later developed in \cite{CKWW,WX}, to apply the global bifurcation to a class of chemotaxis systems without reducing them into a single equation.  It is proved that the steady states form a boundary spike (in the form of a $\delta$-function) or a transition layer (in the form of a step function) if the chemotaxis rate $\chi$ is sufficiently large.  We want to mention that system (\ref{12}) is also able to model the formation of chemotactic bands, which are represented mathematically by travelling wave solutions when the sensitivity is $\Phi=\frac{u}{v}$.  We will not consider the problem in this paper, and for the results on a variants of system (\ref{12}), see the survey papers \cite{HP,Ho,Ho1} for works in this direction and \cite{Wz} for recent developments.  We also want to point out that there are works \cite{BEG, CEV1,CEV2,Ho3, TW1, TW2,WZYH, WW} on Keller--Segel chemotaxis models with two competing species.

System (\ref{12}) has the feature that the total population of cells is preserved since the cellular growth and proliferation have been ignored so far.  This can be a reasonable assumption on the modeling of some chemotaxis systems.  For example, cell proliferation through divisions stops during the aggregation stages of \emph{Dictyostelium Discoideum}.  However Keller--Segel chemotaxis models with the cellular kinetic term have also been proposed and studied and the simplified system of this type over bounded domain $\Omega$ reads as follows
\begin{equation}\label{13}
\left\{
\begin{array}{ll}
u_t=\nabla\cdot(D_1  \nabla u - \chi u\phi(v) \nabla v)+f(u),  & x \in\Omega,~t>0,\\
v_t=D_2 \Delta v -\alpha v+\beta u, & x \in\Omega,~t>0,\\
\frac{\partial u}{\partial \textbf{n}}=\frac{\partial v}{\partial \textbf{n}}=0,&x\in\partial \Omega,~t>0,\\
u(x,0) \geq0,~v(x,0) \geq0,&x\in\Omega,
\end{array}
\right.
\end{equation}
where $D_1$ and $D_2$ are positive constants.  $\alpha> 0$ is a constant that measures the consumption rate of the chemical and $\beta > 0$ interprets phenomenon that the chemicals are released by cells.  This phenomenon is interesting in that the global level patterns of the system emerge through the low level self-organization processes.  The domain boundary $\partial \Omega$ is assumed to be smooth with outer normal $\textbf{n}$.  Here the homogeneous Neumann boundary conditions mean that this bounded region in enclosed and migration and chemical flux .  Moreover the initial data $u(x,0)$ and $v(x,0)$ are assumed to be nonnegative and not identically zero.

To study the effect of the cellular growth on the dynamics of system (\ref{13}), a typical form of $f(u)$ to choose is the logistic function $f(u)= u(\theta-\mu u)$, where $\theta$ and $\mu$ are positive constants.  For $N=1$, it is known from \cite{OY} that all solutions of (\ref{13}) exist globally and are uniformly bounded in time.  For $N=2$, Osaki \emph{et al}. \cite{OTYM} proved the existence of global solutions and obtained an globally exponential attractor of (\ref{13}) provided that the sensitivity function $\phi(v)$ is smooth and has uniformly bounded derivatives up to second order.  For $N\geq3$, Winkler \cite{Wk} has established the unique global solution for all smooth initial data if $\mu$ is sufficiently large.  It seems necessary to point out that, it is demonstrated in \cite{Wk2} that a superlinear growth condition on $f(u)$ may be insufficient to prevent finite time blow-ups for a parabolic-elliptic system of (\ref{13}).

For $f(u)=u(1-u)(u-a)$, $a\in(0,\frac{1}{2})$ with Allee effect, Mimura and Tsujikawa \cite{MT} studied the aggregating patterns of (\ref{13}) when the diffusion rate $D_1$ and the chemotaxis rate $\chi$ are both small enough. Henry \emph{et al.} \cite{HHS} proved the convergence of the solutions and the formation of viscous solutions in the singular limit of a scaling.  Models with other types of $f(u)$ have been considered in \cite{HP,Ho,Ho1,NO,W} and the references therein.

It is the goal of this paper to study the steady states (stationary solutions) of the following general system with a logistic cell growth,
\begin{equation}\label{14}
\left\{
\begin{array}{ll}
u_t=\nabla\cdot(D_1  \nabla u - \chi \Phi(u,v) \nabla v)+u(\theta-\mu u) ,  & x \in\Omega,~t>0,\\
v_t=D_2 \Delta v -\alpha v+h(u), & x \in\Omega,~t>0,\\
\frac{\partial u}{\partial \textbf{n}}=\frac{\partial v}{\partial \textbf{n}}=0, &x\in\partial \Omega,~t>0,\\
u(x,0)=u_0(x)\geq0,~v(x,0)=v_0(x)\geq0,&x\in\Omega.
\end{array}
\right.
\end{equation}
As we see in the aforementioned results, the global-in-time solutions and the exponential attractor of (\ref{14}) have already been obtained by various authors, at least for $\Omega \in \mathbb{R}^N$, $N\leq 2$.  However, there are only few results concerning the steady states of (\ref{14}).  For $D_1=D_2=1$, $\Phi(u,v)=g(u)=h(u)=u$ and $D_1=D_2=\alpha=1$, Tello and Winkler \cite{TW} obtained infinitely many branches of local bifurcating solutions to the stationary problem for all $\mu>0$ if $N\leq 4$ and for all $\mu>\frac{N-4}{N-2}\chi$ if $N>4$.  For $\Phi(u,v)=h(u)=u$, taking $\chi$ as the bifurcation parameter, Kuto \emph{et al.} \cite{KOST} construct local bifurcation branches of strip and hexagonal steady states when the domain $\Omega$ is a rectangle in $\mathbb R^2$.  Moreover, the direction of the pitchfork bifurcation branch is also determined there.  Ma \emph{et al.} \cite{MOW} studied the model with a volume-filling effect, where $\Phi(u)=u(1-u)$ and $h(u)=\beta u$ for $\beta>0$ being a constant.  They carried out the local bifurcation analysis and established a selection mechanism of the wave modes for $\Omega$ being an interval in $\mathbb{R}^1$.  They also showed that the bifurcating solution is stable only at the very branch the principal wave mode of which is a positive integer that minimizes the bifurcation parameter $\chi$.  Very recently model (\ref{14}) over multi-dimension is studied in \cite{JWZ}.  We notice that none of these papers carried out global bifurcation analysis on the steady states of (\ref{14}).

In this paper, we study the stationary problem of the general system (\ref{14}) and we are concerned with the existence and stability of the spatially inhomogeneous positive solutions.  In particular, we are interested in the positive steady states that have interesting patterns such as boundary spikes or transition layers, which can be used to model the cell aggregation phenomenon.

It is easy to see that (\ref{11}) is a stationary system of (\ref{14}) over a one-dimensional domain.  Actually, we introduce the new variables
\[\tilde{D}_1=\frac{D_1}{\mu},\tilde{\chi}=\frac{\chi}{\mu},\tilde{D}_2=\frac{D_2}{\alpha},\tilde{h}(u)=\frac{h(u)}{\alpha},~\bar{u}=\frac{\theta}{\mu},\]
then system (\ref{14}) becomes
\begin{equation}\label{15}
\left\{
\begin{array}{ll}
u_t=\nabla\cdot(D_1  \nabla u - \chi \Phi(u,v) \nabla v)+u(\bar{u}- u) ,  & x \in\Omega,~t>0,\\
v_t=D_2 \Delta v -  v+h(u), & x \in\Omega,~t>0,\\
\frac{\partial u}{\partial \textbf{n}}=\frac{\partial v}{\partial \textbf{n}}=0, &x\in\partial \Omega,~t>0,\\
u(x,0)=u_0(x)\geq0,~v(x,0)=v_0(x)\geq0,&x\in\Omega,
\end{array}
\right.
\end{equation}
where we have dropped the tildes without confusing the reader.  Then we see that (\ref{15}) reduces to (\ref{11}) when $\Omega=(0,L)$.  Note that at least one of $D_1$ and $D_2$ can be scaled to 1 in (\ref{15})while we keep both of them here because they play essential roles in the dynamics of the system as we shall see later on.  Throughout this paper, we make the following assumptions on the sensitivity function $\Phi(u,v)$ for the sake of biological and mathematical considerations:
\begin{equation}\label{16}
\Phi\in C^4(\mathbb{R}\times \mathbb{R},\mathbb{R}),~\Phi(0,0)=0,\Phi(u,v)\geq 0, \Phi_v(u,v)\leq 0, \text{ for all } u,v\geq0,
\end{equation}
and there exists a positive constant $C_1$ such that
\begin{equation}\label{17}
0<\Phi(u,v)\leq C_1u,\text{ for all } u, v>0;
\end{equation}
we also assume that
\begin{equation}\label{18}
h\in C^3(\mathbb{R},\mathbb{R}),~h(0)=0,~h'(u)\geq 0, \text{ for } u\geq0,
\end{equation}
and there exists $C_2>0$ such that
\begin{equation}\label{19}
h(u)\leq C_2u, \text{ for } u\geq0.
\end{equation}
It is easy to see that (\ref{15}) has two equilibria, the trivial one $(0,0)$ and the positive one
\begin{equation}\label{110}
(\bar{u},\bar{v})=(\bar u,h(\bar u)).
\end{equation}
In the absence of chemotaxis (i.e, $\chi=0$), it is well known that $(0,0)$ is unstable and the positive equilibrium $(\bar{u},\bar{v})$ is globally asymptotically stable.  Therefore, system (\ref{15}) does not have any nonconstant steady state when $\chi=0$.  Moreover, from the viewpoint of standard perturbation arguments this conclusion also holds for $\chi>0$ being small, see \cite{Kato,Sim} e.g.  For example, the proof of Theorem 5.1 and Corollary 5.2 in \cite{TW} can be applied to system (\ref{15}) and we have the following result.
\begin{theorem}\label{theorem11}(Theorem 5.1 and Corollary 5.2 \cite{TW})
Assume that conditions (\ref{16})-(\ref{110}) are satisfied.  There exists a positive number $\chi_*$ such that (\ref{15}) has no nonconstant solution if $\chi\in(0,\chi_*)$.
\end{theorem}
Unlike random movements (diffusion), directed movements (chemotaxis) have the effect of destabilizing the spatially homogeneous solutions.  Then spatially inhomogeneous solutions may arise through bifurcation as the homogeneous one becomes unstable.  To study the regime of $\chi$ under which spatial patterns arise, we first implement the standard linear stability analysis of (\ref{15}) at $(\bar{u},\bar{v})$.  Let $(u,v)=(\bar u,\bar v)+(U,V)$, where $U$ and $V$ are small perturbations away from ($\bar u,\bar v$) in $H^2(\Omega)$.  Then we arrive at the following system
\begin{equation*}
\left\{
\begin{array}{ll}
U_t\approx \nabla\cdot(D_1  \nabla U - \chi \Phi(\bar u,\bar v) \nabla V)-\bar{u}U ,  & x \in\Omega,~t>0,\\
V_t \approx D_2 \Delta V - V+h'(\bar u)U, & x \in\Omega,~t>0,\\
\frac{\partial U}{\partial \textbf{n}}=\frac{\partial V}{\partial \textbf{n}}=0, &x\in\partial \Omega,~t>0.\\
\end{array}
\right.
\end{equation*}
According to the standard linearized stability analysis, the stability of $(\bar{u},\bar{v})$ can be determined by the eigenvalues of the following matrix
 \begin{equation}\label{111}
\begin{pmatrix}
 -D_1 \Lambda^2 -\bar{u} & \chi\Phi(\bar{u},\bar{v})\Lambda^2 \\
 h'(\bar{u})              & -D_2 \Lambda^2 -1
  \end{pmatrix} ,
   \end{equation}
where $\Lambda=\Lambda_k>0$, $k=1,2,...$, are the $k$-th eigenvalues of $-\Delta$ on $\Omega$ under the Neumann boundary conditions.

In particular we have that $\Lambda_k=\frac{k\pi}{L}$ for the one-dimensional domain $\Omega=(0,L)$ and the following result gives the linearized instability of $(\bar{u},\bar{v})$ with respect to (\ref{15}).
\begin{proposition}\label{proposition1}
The constant solution $(\bar{u},\bar{v})$ of (\ref{15}) is unstable if
\begin{equation}\label{112}
\chi>\chi_0=\min_{k\in \mathbb{N}^+}\chi_k=\min_{k\in \mathbb{N}^+}\frac{(D_1\big(\frac{k\pi}{L}\big)^2+\bar{u})(D_2\big(\frac{k\pi}{L}\big)^2 +1)}{\Phi(\bar{u},\bar{v})\big(\frac{k\pi}{L}\big)^2 h'(\bar{u})}.
\end{equation}
\end{proposition}

\begin{proof}For $\Omega=(0,L)$, the stability matrix (\ref{111}) becomes
\begin{equation}\label{113}
H_k=
\begin{pmatrix}
 -D_1 \big(\frac{k\pi}{L}\big)^2 -\bar{u} & \chi\Phi(\bar{u},\bar{v})\big(\frac{k\pi}{L}\big)^2 \\
 h'(\bar{u})              & -D_2\big(\frac{k\pi}{L}\big)^2 -1
  \end{pmatrix}.
   \end{equation}
Then $(\bar{u},\bar{v})$ is unstable if $H_k$ has an eigenvalue with positive real part for any $k=1,2,...$.  The characteristic polynomial of (\ref{113}) takes the form
\[p_k(\lambda)=\lambda^2+T_k\lambda+D_k,\]
where \[T_k=\Big(D_1+D_2\Big)\Big(\frac{k\pi}{L}\Big)^2+\bar u+1,\]
\[D_k=\Big(D_1\big(\frac{k\pi}{L}\big)^2+\bar{u}\Big)\Big(D_2\big(\frac{k\pi}{L}\big)^2 +1\Big)-\chi\Phi(\bar{u},\bar{v})\big(\frac{k\pi}{L}\big)^2 h'(\bar{u}),\]
then we see have that $p_k(\lambda)$ has one positive root if $p_k(0)=D_k<0$ from which the instability of $(\bar u,\bar v)$ implies (\ref{112}).  This finishes the proof of this proposition.
\end{proof}

The linear instability of spatially homogeneous solutions is insufficient to prove the existence of spatially inhomogeneous solutions.  However as we have seen above, the chemotaxis term has the effect of destabilizing spatially homogeneous steady states which become unstable if $\chi$ surpasses $\chi_0$, then a stable spatially inhomogeneous steady state of (\ref{15}) may emerge through bifurcations.  Clearly the emergence of spatially inhomogeneous solutions is due to the effect of large chemotaxis rate $\chi$, and we refer to this as cross-diffusion-induced patterns in the sense of Turing's instability.  One of the main contributions of this paper is the detailed bifurcation analysis for system (\ref{15}) at $(\bar{u},\bar{v})$.

The remaining parts of this paper are organized as follows.  In Section \ref{section2}, we formulate (\ref{11}) into a bifurcation problem by taking $\chi$ as the bifurcation parameter and establish infinitely many small amplitude nonconstant positive solutions of (\ref{11}) through local bifurcations-see Theorem \ref{theorem21}.  In particular, we shall see that the bifurcation value $\chi_k$ is the same as that in (\ref{112}).  Then we carry out global bifurcation analysis on the bifurcation branches and show that the solutions of (\ref{11}) on each branch must be strictly positive and all noncompact branches extends to infinity in the $\chi$ direction in Theorem \ref{theorem24}, therefore projection of the bifurcation diagram onto the $\chi$-axis takes the form $[\chi^*,\infty)$ for some $\chi^*\in(\chi_*,\chi_k]$, where $\chi_*$ is obtained in Proposition \ref{proposition1} and $\chi_k$ is the $k$--th bifurcation value.  In Section \ref{section3}, we show that the bifurcation branches are of pitchfork type.  In particular, for $\Phi(u,v)=u$ and $h(u)=\beta u$, $\beta>0$, the stabilities of the small amplitude steady states are determined and fully characterised in terms of system parameters.  See Theorem \ref{theorem31}.  In Section \ref{section4}, the asymptotic behaviors of monotone solutions are investigated as $\chi \rightarrow \chi_\infty \in (\chi_0,\infty]$ and/or the diffusion rate $D_1\rightarrow D_\infty\in[0,\infty]$.  See Theorem \ref{theorem41}.  Finally we include some discussions about our results and propose interesting problems in Section \ref{section5}.

% Results and Discussion can be combined.
\section{Existence of nonconstant positive solutions}\label{section2}
In this section, we study the existence of nonconstant positive solutions of system (\ref{11}).  There are two constant solutions to (\ref{11}): $(0,0)$ and $(\bar u,\bar v)$.  To this end, we shall apply the local bifurcation theory of Crandall and Rabinowitz \cite{CR} and  chemotaxis rate $\chi$ as the bifurcation parameter.  First of all, we define an operator over $\mathcal{X}\times\mathcal{X}\times\mathbb{R} \rightarrow \mathcal{Y} \times \mathcal{Y}$ by
\begin{equation}\label{21}
\mathcal{F}(u,v,\chi)=
 \begin{pmatrix}
(D_1  u' - \chi \Phi(u,v) v' )'+ (\bar{u}-u)u \\
  D_2  v'' -v+h(u)
 \end{pmatrix},
 \end{equation}
where $\mathcal{X}$ is the Hilbert space $H^2_N(0,L)=\{w\in H^2 (0,L)|w'(0)=w'(L)=0\}$ and $\mathcal{Y}=L^2\big(0,L\big)$.   Then we can convert system (\ref{11}) into the following abstract form
\[\mathcal{F}(u,v,\chi)=0, (u,v,\chi)\in \mathcal{X}\times \mathcal{X}\times \mathbb{R}^+.\]

Obviously the operator $\mathcal{F}$ is a continuously differentiable mapping from $\mathcal{X}\times \mathcal{X}\times \mathbb{R}$ to $\mathcal{Y} \times \mathcal{Y}$.  In order to apply the local bifurcation theory from Crandall and Rabinowitz \cite{CR}, we collect the following facts about the operator $\mathcal{F}(u,v,\chi)$:
\begin{enumerate}
  \item[Fact 1.]    $\mathcal{F}(\bar{u},\bar{v},\chi)=0$ for any $\chi \in \mathbb{R}^+$, where $( \bar{u},\bar{v})$ is the constant equilibrium;
  \item[Fact 2.]   for any fixed $(u_0,v_0) \in \mathcal{X} \times \mathcal{X}$, the Fr$\acute{\text{e}}$chet derivative of $\mathcal{F}$ is given by
%  \begin{equation}\label{22}
%  D_{(u,v)}\mathcal{F}(u_0,v_0,\chi)(u,v)=
%  \begin{pmatrix}
%  \Big(D_1 u' - \chi\big((\Phi_u(u_0,v_0)u+\Phi_v(u_0,v_0)v\big)v_0'+\Phi(u_0,v_0) v'\Big)' + (\bar u-2u_0)u \\
%  D_2 v'' -v+ h'(u_0) u
%  \end{pmatrix};
%  \end{equation}
  \begin{align}\label{22}
 &D_{(u,v)}\mathcal{F}(u_0,v_0,\chi)(u,v) \\
=&\begin{pmatrix}
  \Big(D_1 u' - \chi\big((\Phi_u(u_0,v_0)u+\Phi_v(u_0,v_0)v\big)v_0'+\Phi(u_0,v_0) v'\Big)' + (\bar u-2u_0)u\\
  D_2 v'' -v+ h'(u_0) u
  \end{pmatrix}\nonumber
  \end{align}
and in particular, for $(u_0,v_0)=(\bar u,\bar v)$, we have that
  \begin{equation}\label{23}
  D_{(u,v)}\mathcal{F}(\bar u,\bar v,\chi)(u,v)=
  \begin{pmatrix}
  D_1 u'' - \chi \Phi(\bar u,\bar v) v'' -\bar uu \\
  D_2 v''-v + h' (\bar u)u
  \end{pmatrix};
  \end{equation}

  \item[Fact 3.]  for any fixed $(u_0,v_0)\in \mathcal{X}\times \mathcal{X}$, $D_{(u,v)}\mathcal{F}(u_0,v_0,\chi)$ : $\mathcal{X} \times \mathcal{X}\to \mathcal{Y} \times \mathcal{Y}$ is a \emph{Fredholm} operator with zero index.
\end{enumerate}
Fact 1 and fact 2 can be verified by straightforward calculations.  To show fact 3, we rewrite (\ref{22}) as
\[D_{(u,v)} \mathcal{F}(u_0,v_0,\chi)(u,v)=I_1{u \choose v}''+ I_2{u \choose v}'+I_3{u \choose v}, \]
where \[I_1=
\begin{pmatrix}
 D_1  & -\chi\Phi(u_0,v_0) \\
  0 & D_2
  \end{pmatrix},~I_2=\begin{pmatrix}
  -\chi\Phi_u(u_0,v_0)v'_0~~~ & -\chi\big(\Phi_u(u_0,v_0)u'_0+2\Phi_v(u_0,v_0)v'_0\big) \\
  0 ~~~& 0
  \end{pmatrix} \]
and
\[I_3=\begin{pmatrix}
-\chi\big(\Phi_u(u_0,v_0) v'_0\big)'+\bar u-2u_0 ~~~& -\chi\big(\Phi_v(u_0,v_0) v'_0 \big)' \\
  h'(u_0)  ~~~& -1
  \end{pmatrix}.\]
Then we know that $D_{(u,v)}\mathcal{F}(u_0,v_0,\chi)$ is a linear and compact operator according to the standard elliptic regularity and Sobolev embedding theorems.  Moreover, we see that matrix $I_1$ defines the principal part of the elliptic operator $D_{(u,v)} \mathcal{F}(u_0,v_0,\chi)$ and it has two positive eigenvalues, then according to Theorem 4.4 or case 3 of Remark 2.5 in Shi and Wang \cite{SW}, this operator satisfies the \emph{Agmon's condition}.  Moreover it is a \emph{Fredholm} operator with zero index according to Corollary 2.11 or Remark 3.4 of Theorem 3.3 in \cite{SW}, from which Fact 3 follows.  This proves all the facts needed for the local bifurcation analysis.

Now we identify potential candidates of the bifurcation values $\chi$.  In order to let the bifurcations occur at the equilibrium $(\bar u,\bar v,\chi)$, we need the implicit function theorem to fail there, therefore the mapping $D_{(u,v)}\mathcal{F}(\bar u,\bar v,\chi)$ in (\ref{23}) must have a nontrivial kernel, i.e,
\[\mathcal{N}(D_{(u,v)}\mathcal{F}(\bar u,\bar v,\chi))\ne \{\textbf{0}\},\]
where $\mathcal{N}$ denotes the null set.  We argue by contradiction.  If not, we choose $(u,v)\in D_{(u,v)}\mathcal{F}(\bar u,\bar v,\chi)$ and write their eigen-expansions as
\begin{equation}\label{24}
u(x)=\sum_{k=0}^{\infty}  \bar u_k(x), v(x)=\sum_{k=0}^{\infty} \bar  v_k(x),
\end{equation}
where \[\bar u_k=T_k \cos \frac{k\pi x}{L},   \bar v_k=S_k \cos\frac{k\pi x}{L}\] and $T_k$ and $S_k$ are constants.  Substituting (\ref{24}) into (\ref{23}), we have that
 \begin{equation}\label{25}
\begin{pmatrix}
 -D_1 \big(\frac{k\pi}{L}\big)^2 -\bar u  ~~~& \chi\Phi(\bar u,\bar v)\big(\frac{k\pi}{L}\big)^2 \\
 h'(\bar u)             ~~~  & -D_2\big(\frac{k\pi}{L}\big)^2 -1
  \end{pmatrix} \begin{pmatrix}
  T_k \\
  S_k
   \end{pmatrix}=0,
   \end{equation}
then our assumption above requires that system (\ref{25}) admits at least one nonzero solution, therefore its coefficient matrix must be singular and we arrive at the following identity
\begin{equation}\label{26}
\begin{vmatrix}
 -D_1 \big(\frac{k\pi}{L}\big)^2 -\bar u  & \chi\Phi(\bar u,\bar v)\big(\frac{k\pi}{L}\big)^2 \\
h'(\bar u)              & -D_2\big(\frac{k\pi}{L}\big)^2 -1
  \end{vmatrix}=0.
  \end{equation}
From straightforward calculations, we obtain the following potential bifurcation value of $\chi$, which will be denoted by $\chi_k$ from now on
 \[ \chi_k=\frac{(D_1\big(\frac{k\pi}{L}\big)^2+\bar u)(D_2\big(\frac{k\pi}{L}\big)^2 +1)}{\Phi(\bar u,\bar v)\big(\frac{k\pi}{L}\big)^2 h'(\bar u)}>0,  k\in \mathbb{N}^+.\]
Note that $\chi_k$ here is the same as that given in (\ref{112}).  $k=0$ can be easily excluded in (\ref{26}) and for each $k\in\mathbb N^+$ we see that $\text{dim}\Big(\mathcal{N}(D_{(u,v)}\mathcal{F}(\bar u,\bar v,\chi_k)) \Big)=1$ and in particular
\begin{equation}\label{27}
\mathcal{N}(D_{(u,v)}\mathcal{F}(\bar u,\bar v,\chi_k))=\text{span}\Big\{(\bar u_k(x),\bar v_k(x))\Big\},
\end{equation}
where
\begin{equation}\label{28}
(\bar u_k(x),\bar v_k(x))=\Big(Q_k\cos \frac{k\pi x}{L}, \cos \frac{k\pi x}{L}\Big),~ Q_k=\frac{D_2\big(\frac{k\pi}{L}\big)^2+1}{h'(\bar u)},~k\in  \mathbb{N}^+.
\end{equation}
Before proceeding our analysis, we remark that the local bifurcation does not occur at $(0,0)$.  Actually, if $(\bar{u},\bar{v})=(0,0)$, the coefficient matrix in (\ref{25}) is nonsingular and we must have that $T_k=S_k=0$ for all $k$.  Therefore $\mathcal{N}(D_{(u,v)}\mathcal{F}(0,0,\chi))=\{\textbf{0}\}$ for all $\chi \in \mathbb{R}$ and this contradicts the necessary condition for bifurcation to occur at $(0,0,\chi)$.  Now we are ready to prove the following theorem, which is the first bifurcation result of our paper.
\begin{theorem}\label{theorem21}
Suppose that conditions (\ref{16}) and (\ref{18}) are satisfied.  We assume that
\begin{equation}\label{29}
\bar u\neq j^2 k^2 D_1D_2\left(\frac{\pi}{L}\right)^4\text{ for all positive integers }j\neq k.
\end{equation}
Then for each $k\in \mathbb N^+$, a branch of spatially inhomogeneous solutions of (\ref{11}) bifurcate from $(\bar u,\bar v)$ at $\chi=\chi_k$.  Moreover, there exist a small positive constant $\delta$ and continuous functions $s\in(-\delta,\delta):\to \chi_k(s)\in \mathbb R^+$
and $s\in(-\delta,\delta):\to(u_k(s,x),v_k(s,x))\in\mathcal{X}\times \mathcal{X}$ such that $\chi_k(0)=\chi_k$ and the bifurcation branches around $(\bar u,\bar v,\chi_k)$ can be parameterized as
\[\chi_k(s)=\chi_k+O(s),(u_k(s,x),v_k(s,x))=(\bar u,\bar v)+s(Q_k,1)\cos \frac{k\pi x}{L}+s(\xi_k(s),\eta_k(s) ),\]
where $(\xi_k(s),\eta_k(s))$ is an element in a closed complement of $\mathcal{N}(D_{(u,v)}\mathcal{F}(\bar u,\bar v,\chi_k))$ in $\mathcal{X} \times\mathcal{X}$ with $(\xi_k(0),\eta_k(0))=(0,0)$; furthermore $(u_k(s,x),v_k(s,x),\chi_k(s))$ solves system (\ref{11}) and all nonconstant solutions of (\ref{11}) around $(\bar u,\bar v,\chi_k)$ must stay on the curve
\[\Gamma_k(s):=s\in(-\delta,\delta) \rightarrow (u_k(s,x),v_k(s,x),\chi_k(s))\in \mathcal{X}\times \mathcal{X}\times \mathbb{R}.\]
\end{theorem}
\begin{proof}
To make use of the local bifurcation theory in \cite{CR}, we have verified all but the so-called \emph{transversality condition}:
\[\frac{d}{d\chi}D_{(u,v)}\mathcal{F}(\bar u,\bar v,\chi)(\bar u_k,\bar v_k)|_{\chi=\chi_k}\notin \mathcal{R}(D_{(u,v)}\mathcal{F}(\bar u,\bar v,\chi_k)).\]
We argue by contradiction and suppose there exists $(\widetilde u,\widetilde v)$ such that the transversality condition above fails.  Obviously
\begin{equation}\label{210}
 \frac{d}{d\chi}D_{(u,v)}\mathcal{F}(\bar u,\bar v,\chi)(\bar u_k,\bar v_k)|_{\chi=\chi_k}=
{-\Phi(\bar u,\bar v)   \bar v_k'' \choose 0},
\end{equation}
where $\bar u_k=Q_k \cos \frac{k\pi x}{L}$ and $\bar v_k=\cos \frac{k\pi x}{L}$ as defined in (\ref{28}),
then we obtain that
\begin{equation}\label{211}
\left\{
\begin{array}{ll}
D_1 \widetilde u'' - \chi_k \Phi(\bar u,\bar v) \widetilde v''- \bar{u} \widetilde u=-\Phi(\bar u,\bar v) \bar v_k'',  & x \in (0,L),\\
D_2 \widetilde v'' -\widetilde v + h' (\bar u)\widetilde u ,=0, & x \in (0,L),\\
\widetilde u(x)=\widetilde v(x)=0,& x=0,L.\\
\end{array}
\right.
\end{equation}
Similar as the analysis above, we expand $\tilde u$ and $\tilde v$ as
\[\tilde{u}=\sum_{k=0}^\infty \tilde T_k \cos \frac{k\pi x}{L}, \tilde{v}=\sum_{k=0}^\infty \tilde S_k \cos \frac{k\pi x}{L},\]
and substitute the series into (\ref{211}) to obtain the following system
\begin{equation}\label{212}
\begin{pmatrix}
 -D_1 (\frac{k\pi}{L})^2 -\bar u   ~~~& \chi_k\Phi(\bar u,\bar v)\big(\frac{k\pi}{L}\big)^2 \\
  h'(\bar u)               ~~~ & -D_2\big(\frac{k\pi}{L}\big)^2 -1
  \end{pmatrix} \begin{pmatrix}
  \widetilde T_k \\
  \widetilde S_k
   \end{pmatrix}=\begin{pmatrix}
 \Phi(\bar u,\bar v)\big(\frac{k\pi}{L}\big)^2  \\
  0
  \end{pmatrix}.
   \end{equation}
Now we see from (\ref{26}) that the coefficient matrix is singular and  the right hand side of (\ref{212}) is not in the range of the matrix, therefore system (\ref{211}) is unsolvable.  We reach a contradiction and the \emph{transversality condition} is verified.  Finally, we need that $\chi_k\neq\chi_j$ for all $j\neq k$, i.e,
\[\frac{(D_1\big(\frac{k\pi}{L}\big)^2+\bar u)(D_2\big(\frac{k\pi}{L}\big)^2 +1)}{\Phi(\bar u,\bar v)\big(\frac{k\pi}{L}\big)^2 h'(\bar u)}\neq\frac{(D_1\big(\frac{j\pi}{L}\big)^2+\bar u)(D_2\big(\frac{j\pi}{L}\big)^2 +1)}{\Phi(\bar u,\bar v)\big(\frac{j\pi}{L}\big)^2 h'(\bar u)},\]
which is equivalent as (\ref{29}) as we can show through straightforward calculations.
\end{proof}
From the local bifurcation analysis, we are able to obtain nonconstant positive solutions with small amplitudes around $(\bar u,\bar v,\chi_k)$ for all $k\in\mathbb{N}^+$.  Moreover, we know from Proposition \ref{proposition1} that, the equilibrium $(\bar u,\bar v)$ is stable for all $\chi<\min_{k\in\mathbb{N}^+}\chi_k$ and it becomes unstable if $\chi>\min_{k\in\mathbb{N}^+}\chi_k$, which is exactly the location where the bifurcation occurs.  Then we see that the instability of homogeneous steady state and pattern formations are driven by the cross-diffusion (the chemotaxis term) in the sense of Turing's instability.

According to the local bifurcation analysis above, we have established nonconstant positive solutions of (\ref{11}), which are small perturbations of ($\bar{u},\bar{v}$).  We now proceed to extend the local curves $\Gamma_k(s)$ by the global bifurcation theory for nonlinear \emph{Fredholm} mappings from \cite{CR,PR} and the recent version developed in \cite{SW}.  The first step of our analysis is to present the following \emph{a priori} estimates on the solutions of (\ref{11}).
\begin{lemma}\label{lemma22}
Assume that condition (\ref{19}) is satisfied.  Let $(u,v)$ be any positive solution to the boundary value problem (\ref{11}).  Then
\begin{equation}\label{213}
\bar u \Vert u \Vert_{L^1(0,L)}=\Vert u \Vert^2_{L^2(0,L)} \leq \bar u^2 L
\end{equation}
and \[\min_{[0,L]} u\leq \bar u\leq \max_{[0,L]} u; \]  moreover, there exists a positive constant C independent of $D_1$ and $\chi$ such that
\begin{equation}\label{214}
\Vert v \Vert_{H^2(0,L)}\leq C.
\end{equation}
\end{lemma}

\begin{proof}
Integrating the first equation of (\ref{11}) over $(0,L)$ by parts, we have from the H\"older's inequality that
\[\Vert u \Vert^2_{L^2(0,L)}=\bar u \Vert u \Vert_{L^1(0,L)} \leq \bar u \Vert u \Vert_{L^2(0,L)} \sqrt{L}, \]
and it follows that $\Vert u \Vert_{L^2(0,L)} \leq \bar u\sqrt{L}$.  On the other hand, we can also see that
\[ \int_0^L (\bar u-u)udx =0,\]
then we must have that either $(\bar u-u)u\equiv 0$ or $(\bar u-u)u$ changes sign over $(0,L)$, hence $\min_{[0,L]} u\leq \bar u \leq \max_{[0,L]} u $ in either case.  On the other hand, by applying the elliptic regularity theory to the second equation of (\ref{11}) and using (\ref{19}), we can show that $\Vert v \Vert_{H^2(0,L)}$ is uniformly bounded by a positive constant $C$.  This finishes the proof of the lemma.
\end{proof}

\begin{lemma}\label{lemma23}
Assume that conditions (\ref{16})-(\ref{19}) are satisfied.  Let $(u,v)$ be any positive solution of (\ref{11}).  Then we have that
\begin{equation}\label{215}
u(x)\leq u(L) e^{\frac{\chi}{D_{1}}\Vert \frac{\Phi(u,v)}{u}  v'\Vert_\infty (L-x)}+\frac{\bar u^2L}{D_1} \int_x^L e^{-\frac{\chi}{D_{1}}\Vert \frac{\Phi(u,v)}{u}  v'\Vert_\infty(x-y)}dy.
\end{equation}

\end{lemma}
\begin{proof}
We integrate the $u$-equation in (\ref{11}) over $(x,L)$ and obtain that
\begin{equation}\label{216}
u'(x)-\frac{\chi}{D_{1}}\Phi(u,v)v'=\frac{1}{D_{1}}\int_x^L (\bar u-u)u dy \geq -\frac{1}{D_{1}}\int_x^L u^2dy\geq -\frac{\bar u^2L}{D_{1}},
\end{equation}
where the last inequality follows from (\ref{213}).  Note that $\Vert \frac{\Phi(u,v)}{u}  v'\Vert_{\infty}$ is bounded because of (\ref{17}) and (\ref{214}).  Then we have from (\ref{216}) that
\begin{equation}\label{217}
u'(x)+\frac{\chi}{D_{1}}\Vert \frac{\Phi(u,v)}{u}  v'\Vert_{L^\infty}u\geq -\frac{\bar u^2L}{D_{1}},
\end{equation}
and the Gr\"onwall's inequality implies (\ref{215}).
\end{proof}

We see from (\ref{215}) and the standard elliptic regularity that $\Vert u\Vert_{H^2}$ is bounded if both $u(L)$ and $\chi/D_1$ are finite.  Moreover, according to the uniform boundedness of $\Vert v\Vert_{H^2}$ and the Sobolev embedding, we have that $\forall \gamma\in(0,\frac{1}{2})$, $\Vert v\Vert_{C^{1+\gamma}}$ is uniformly bounded for all $\chi\in(0,\infty)$.

Now we carry the global bifurcation analysis on the local bifurcation curve $\Gamma_k(s)$ established in Theorem \ref{theorem21}.  In particular, we shall show that all solutions on the continuum of each branch must be strictly positive on $[0,L]$.  In the rest of this section, we shall drop the subindex $k$ and denote $\Gamma(s)$ as $\Gamma_k(s)$ without confusing our reader.

For each $k\in\mathbb N^+$, let $\Gamma_u=\{(u_k(s,x),v_1(s,x),\chi_k(s)\vert s\in(0,\delta)\}$ be the upper branch and $\Gamma_l=\{(u_1(s,x),v_1(s,x),\chi_k(s)\vert s\in(-\delta,0)\}$ be the lower branch of the bifurcation curve $\Gamma(s)$ near the bifurcation point ($\bar{u},\bar{v},\chi_k$) respectively.  On the other hand, we can easily show from the strong maximum principle and Hopf's lemma that $u(x)\geq0$ and $v(x)\geq0$ for $x\in [0,L]$ for all solutions of system (\ref{11}).  Thus we let $V=(\mathcal{X} \times \mathcal{X} \times \mathbb{R}^+)\cap \{(u,v,\chi) \vert u(x) \geq 0, v(x)\geq 0,~x\in[0,L]\}$ and consider the problem (\ref{11}) in this cone.  Denoting the solution set of (\ref{11}) by $S=\big\{(u,v,\chi)\in V:F(u,v,\chi)=0, (u,v)\neq(\bar{u},\bar{v}) \big\}$ and $\bar{S}$ the closure of $S$, we readily see that $\bar{S}$ is not empty since $\Gamma(s)$ is contained in $\bar{S}$.

Let $\mathcal{C}$ be a connected component (maximal connected subset) of $\bar{S}$ and $\mathcal{C}^+$ be the connected component of $\mathcal{C}\backslash\{\Gamma_l\cup(\bar u,\bar v,\chi_k) \}$ that contains $\Gamma_u$ (resp. $\mathcal{C}^-$ be the connected component of $\mathcal{C}\backslash\Gamma_u \cup(\bar u,\bar v,\chi_k) \}$ that contains $\Gamma_l$), then we have from Theorem 4.4 in \cite{SW} that each of the sets $\mathcal{C}^+$ and $\mathcal{C}^-$ satisfies one of the following alternatives: (i) it is not compact in $V$; (ii) it contains a point $(\bar{u},\bar{v},\chi^*)$ with $\chi^*\neq \chi_k$; or (iii) it contains a point $(\bar{u}+u,\bar{v}+u,\chi)$, where $(u,v)\neq (0,0)$ and $(u,v)\in \mathcal{Z}$, where $\mathcal{Z}$ is a closed complement of $\mathcal{N}\Big(D_{(u,v)} \mathcal{F}(\bar{u},\bar{v},\chi_k) \Big)$ in $\mathcal{X} \times \mathcal{X}$.  Without loss of our generality, we take from now on
\begin{equation}\label{218}
\mathcal{Z}=\Big\{(u,v)\in \mathcal{X}\times \mathcal{X}|\int_0^L\big(Q_ku(x)+v(x)\big)\cos\frac{k\pi x}{L}=0\Big\},
\end{equation}
where $Q_k=\frac{D_2(\frac{k\pi}{L})^2 +1}{h'(\bar{u})}$.  We will show that if $\mathcal{C}^+$ (also $\mathcal{C}^-$) is noncompact then it extends to infinity in the positive direction of $\chi$-axis.  Therefore, the projection of the continuum of the solution set $\mathcal{C}^+$ (and also $\mathcal{C}^-$) takes the form $[\tilde \chi,\infty)$ for some $\tilde \chi\in (\chi_*,\chi_k]$, where $\chi_*$ is obtained in Theorem \ref{theorem11}.  Moreover, all solutions on $\mathcal{C}^+$ (and also $\mathcal{C}^-$) must be strictly positive on $[0,L]$.  Without loss of our generality, we study only the upper branch $\mathcal{C}^+$ and we are now in a position to present another main result of this paper.
\begin{theorem}\label{theorem24}
Suppose that the conditions in Theorem \ref{theorem21} are satisfied.  Then all solutions in each bifurcation branch $\Gamma_k(s)$ are strictly positive on $[0,L]$.  Moreover the non-compact continuum of each branch can only extend to infinity in the positive $\chi$-axis direction.
\end{theorem}
\begin{proof}  As a matter of fact, we will show that the elements on $\mathcal{C}^+$ satisfy the properties described in the theorem while the same conclusions can be made about $\mathcal C^-$.  We first prove that $(u,v,\chi)$ stays strictly positive on $\mathcal{C}^+$  for $x\in [0,L]$.

To this end, we introduce the set of positive functions $\mathcal{P}=\{(u,v)\in\mathcal{X} \times \mathcal{X}\vert u(x)>0,v(x)>0, x\in[0,L] \}$ and we want to show that $\mathcal{C}^+\subset \mathcal{P} \times \mathbb{R}^+$ (notice that their intersection is not empty because at least it contains the portion of $\Gamma_u(s)$ near the bifurcation point $(\bar u,\bar v,\chi_k)$ (i.e, with $s\in(0,\delta)$)).  If this fails, since $\mathcal{C}^+$ is connected and $\mathcal{P} \times \mathbb{R}^+$ is open, there exists a solution $(u,v,\chi)\in \mathcal{C}^+ \times \partial(\mathcal{P} \times \mathbb{R}^+)$  to (\ref{11}) such that $u, v\geq 0$ on $[0,L]$ and either $u(x)=0$ or $v(x)=0$ somewhere over $[0,L]$, or $\chi=0$.  If $\chi=0$, system (\ref{11}) becomes
\begin{equation}\label{219}
\left\{
\begin{array}{ll}
D_1u''+(\bar u-u)u=0,  & x \in(0,L),\\
D_2 v'' -v + h(u)=0, & x \in(0,L),\\
u'(x)=v'(x)=0,&x=0,L.
\end{array}
\right.
\end{equation}
It is known from our discussions in the introduction section that system (\ref{219}) admits only constant solution $(0,0)$ or $(\bar u,\bar v)$.  Moreover, we know that bifurcation does not occur at $(0,0)$.  Therefore $(u,v)\equiv (\bar u,\bar v)$ and $\chi=0$ must be a bifurcation value.  However this is impossible since all bifurcation values take the form $\chi_k$ in (\ref{112}) which must be positive as we have shown in Theorem \ref{theorem21}.  If $v(x_0)=0$ for some $x_0\in[0,L]$, we apply the strong maximum principle and Hopf's lemma to the following problem
\begin{equation}\label{220}
\left\{
\begin{array}{ll}
D_2 v'' -v =- h(u)\leq0, & x \in(0,L),\\
v'(x)=0,&x=0,L,
\end{array}
\right.
\end{equation}
then we have that $v(x)\equiv 0$ on $[0,L]$ and it follows from the $u$-equation in (\ref{11}) that $u(x)\equiv0$.  This again reaches a contradiction since bifurcation does not occur at $(0,0)$.  Therefore we must have that $v(x)>0$ on $[0,L]$.  Similarly we apply the strong maximum principle and Hopf's lemma to the $u$-equation in (\ref{11}), which is equivalent to
\begin{equation}\label{221}
\left\{
\begin{array}{ll}
D_1u''- \chi \Phi_u(u,v) v'u'-\Big(\chi \frac{\Phi(u,v)}{u}v''+u-\bar u\Big)u=\chi\Phi_v(u,v)(v')^2\leq0,&x\in(0,L),\\
u'(x)=0,&x=0,L,
\end{array}
\right.
\end{equation}
where coefficients of $u'$ and $u$ are bounded because $\frac{\Phi(u,v)}{u}$ is bounded thanks to (\ref{17}).  Then we must have $u(x)>0$ on $[0,L]$ and this completes the proof of the positivity part.

Suppose that $\mathcal{C}^+$ is unbounded in $\mathcal{X} \times \mathcal{X} \times \mathbb{R}$.  We shall show that is can only extend to infinity in the positive direction $\chi$-axis.  According to Lemma \ref{lemma22} and Lemma \ref{lemma23},  both $u$ and $v$ are bounded in $\mathcal{X}=H^2(0,L)$ given a finite $\chi>0$, hence $\mathcal{C}^+$ must extend to infinity along the $\chi$-axis.  On the other hand, suppose that $\mathcal{C}^+$ extends to the negative direction of $\chi$-coordinate, then it must cross $\chi=0$, for which (\ref{11}) no non-constant positive solution which implies that 0 is a bifurcation value, however this is impossible as we have shown above that any bifurcation value must be $\chi_k$ given by (\ref{112}).  Therefore $\mathcal{C}^+$ must extend to infinity in the positive direction of $\chi$-coordinate and its project onto the $\chi$-axis takes the form $[\chi_0,\infty)$ for some $\chi_0\leq \chi_k$.  This completes the proof of Theorem \ref{theorem24}.
\end{proof}

According to the discussions above, $\mathcal{C}^+$ satisfies one of the following alternatives: (i) it is not compact in $V$; (ii) it contains a point $(\bar{u},\bar{v},\chi^*)$ with $\chi^*\neq \chi_k$; or (iii) it contains a point $(\bar{u}+u,\bar{v}+u,\chi)$, where $(u,v)\neq (0,0)$ and $(u,v)\in \mathcal{Z}$, which is defined in (\ref{218}).  When there is no cellular growth, \cite{WX} ruled out the last two cases for the first branch by showing that all bifurcating solutions on the continuum stay monotone, i.e., $u'(x)<0$ and $v'(x)<0$ on $(0,L)$.  To apply their topology argument we proceed as follows.  Define the set $\mathcal{P}^+=\{(u,v)\in \mathcal{X} \times \mathcal{X} \vert u'(x)<0, v'(x)<0,x\in (0,L)\}$ and we need to show that $\mathcal{C}^+\subset \mathcal{P}^+ \times \mathbb{R}^+ $.  Apparently $\mathcal{C}^+ \cap (\mathcal{P}^+ \times \mathbb{R}^+) \neq \emptyset$.  Since $\mathcal{C}^+$ is a connected subset of $\mathcal{X} \times \mathcal{X} \times \mathbb{R}$, it is sufficient to show that $\mathcal{C}^+ \cap (\mathcal{P}^+ \times \mathbb{R}^+)$ is both open and closed with respect to the topology of $\mathcal{C}^+$.  To show the openness, we take some $(\tilde{u},\tilde{v},\tilde{\chi})\in \mathcal{C}^+ \cap (\mathcal{P}^+ \times \mathbb{R}^+)$ and assume that there exists a sequence $\{(\tilde{u}_n,\tilde{v}_n,\tilde{\chi}_n)\}$ in $\mathcal{C}^+$ that converges to $(\tilde{u},\tilde{v},\tilde{\chi})$ in the norm of $\mathcal{X} \times \mathcal{X} \times \mathbb{R}$.  Then we have from the standard elliptic regularity theories that $(\tilde{u}_n,\tilde{v}_n) \rightarrow (\tilde{u},\tilde{v})$ in $C^2([0,L]) \times C^2([0,L])$.  Differentiating the $v$-equation in (\ref{11}), we have that
\begin{equation}\label{222}
\left\{
\begin{array}{ll}
D_2 (\tilde{v}')'' - \tilde{v}' =- h'(\tilde{u})\tilde{u}'\geq0, & x \in(0,L),\\
\tilde{v}'(0)=\tilde{v}'(L)=0.
\end{array}
\right.
\end{equation}
Then we conclude from Hopf's lemma that
\begin{equation}\label{223}
\tilde{v}''(L)>0>\tilde{v}''(0).
\end{equation}
This second order non-degeneracy at the boundary, together with the fact $\tilde{v}'<0$, implies that $\tilde{v}'_n<0$ on $(0,L)$ for large $n$.  Similarly we can show that $\tilde{u}''(0)<0<\tilde{u}''(L)$ and again the second order non-degeneracy implies that $\tilde{u}'_n<0$ on $(0,L)$ for large $n$.  This completes the proof of the openness.  To show the closedness of $\mathcal{C}^+ \cap (\mathcal{P}^+ \times \mathbb{R}^+)$ in $\mathcal{C}^+$, we take a sequence $\{(\tilde{u}_n,\tilde{v}_n,\tilde{\chi}_n)\} \in \mathcal{C}^+ \cap (\mathcal{P}^+ \times \mathbb{R}^+)$ and assume that there exists $(\tilde{u},\tilde{v},\tilde{\chi})$ such that $\{(\tilde{u}_n,\tilde{v}_n,\tilde{\chi}_n)\}  \rightarrow (\tilde{u},\tilde{v},\tilde{\chi})$ in the topology of $\mathcal{C}^+$.  Again, by elliptic regularity theory we have that, $\{(\tilde{u}_n,\tilde{v}_n)\}  \rightarrow (\tilde{u},\tilde{v})$ in $C^2([0,L]) \times C^2([0,L])$ and $\tilde{u}'(x)\geq0$, $\tilde{v}'(x)\geq0$ on $(0,L)$.  It is sufficient to show that $\tilde{u}'(x)>0$ and $\tilde{v}'(x)>0$ on $(0,L)$ and we first prove the latter one by a contradiction argument.  If $\tilde{v}'(x)=0$ for $x_0\in(0,L)$, we apply the strong maximum principle and Hopf's lemma to (\ref{222}) and have that $\tilde{v}'(x)\equiv0$.  Then we see that the $u$-equation becomes $D_1 \tilde{u}''(x)+ (\bar u-\tilde{u})\tilde{u}=0$, $\tilde{u}'(0)=\tilde{u}'(L)=0$, which implies that $\tilde{u}'(x)=0$.  Thus $(\tilde{u},\tilde{v})\equiv(0,0)$ or $(\tilde{u},\tilde{v})\equiv(\bar u,\bar v)$.  The first case is impossible, since we have shown that bifurcation can not occur at $(0,0)$.  Then $(\tilde{u},\tilde{v})\equiv(\bar u,\bar v)$ and $\tilde{\chi}$ is a bifurcation value thus equals $\chi_k$ for some $k\geq1$.  We know that $k=1$ is impossible since $(\bar u,\bar v,\chi_1)\not\in\mathcal{C}^+$.  Moreover, $k\geq2$ is also impossible since $(\tilde{u}_n,\tilde{v}_n,\tilde{\chi}_n)$ around the bifurcation point $(\bar u,\bar v,\chi_k), k\geq2$, satisfy the formula in Theorem \ref{theorem21} and must be non-monotone around ($\bar u,\bar v,\chi_k$), which is a contradiction to our assumption that $\tilde{u}'\leq0$ and $\tilde{v}'\leq0$ on $(0,L)$.  However, one can not show that $\tilde{u}'<0$ of ($0,L$).  Therefore one needs a novel approach than \cite{WX} to this end.  Moreover it is interesting and also important to study the nodal profiles of the non-monotone bifurcating solutions.  However these questions are out of the scope of this paper.

\section{Stability analysis of the bifurcating solutions}\label{section3}
In this section, we investigate the stability or instability of the spatially inhomogeneous solution $(u_k(s,x),v_k(s,x))$ that bifurcates from $(\bar{u},\bar{v})$ at $\chi=\chi_k$.  For this purpose, we apply the classical the linearized stability results of Crandall and Rabinowitz in \cite{CR2} through the analysis of the spectrum of system (\ref{11}).  Stability here refers to that of the inhomogeneous patterns taken as an equilibrium to (\ref{15}).  First of all, we determine the direction in which the bifurcation curve $\Gamma(s)$ turns around $(\bar{u},\bar{v},\chi_k)$.

\subsection{Bifurcation of pitchfork type}We recall from Theorem 1.7 in \cite{CR} that for any $s\in (-\delta,\delta)$,
$\Big(u_k(s,x)-\bar u-sQ_k \cos\frac{k\pi x}{L},v_k(s,x)-\bar v-s\cos\frac{k\pi x}{L}\Big) \in\mathcal{Z}$, where $\mathcal{Z}$ is defined as in (\ref{218}).  Furthermore, if $\Phi(u,v)$ is $C^5$-smooth, then $\mathcal{F}$ defined in (\ref{21}) is $C^4$-smooth.  According to Theorem 1.18 of \cite{CR}, $(u_k,v_k,\chi_k)$ are $C^3$-smooth functions of $s$ and we can write the following expansions:
\begin{equation}\label{31}
  \left \{
\begin{array}{ll}
u_k(s,x)=\bar u + sQ_k\cos\frac{k\pi x}{L}+s^2\psi_1 +s^3\psi_2  +o(s^3),\\
v_k(s,x)=\bar v + s\cos\frac{k\pi x}{L}+s^2\varphi_1 +s^3\varphi_2  +o(s^3),\\
\chi_k(s)= \chi_k+ \mathcal{K}_2s+\mathcal{K}_3s^2+o(s^2),\\
\end{array}
\right.
\end{equation}
where $(\psi_i,\varphi_i)\in \mathcal{Z}$ for $i=1,2$ and $\mathcal{K}_2$, $\mathcal{K}_3$ are constants.  Note that the $o(s^3)$ terms in $u_k(s,x)$ and $v_k(s,x)$ are measured in the $H^2$-norms.

As we shall see the coming analysis that, if $\mathcal{K}_2\neq0$, the sign of $\mathcal{K}_2$ determines the stability of $(u_k(s,x),v_k(s,x))$, and if $\mathcal{K}_2=0$, we need to determine the sign of $\mathcal{K}_3$, and so on so forth.  Now we write each component of the $u$-equation into a series of $s$ and then obtain the following identities from straightforward calculations,
\begin{eqnarray}\label{32}
D_1u''&=&-D_1\left(\frac{k\pi}{L}\right)^2s Q_k\cos\frac{k\pi x}{L}+s^2D_1\psi''_1+s^3D_1\psi_2''+o(s^3),\\
v'&=&-\left(\frac{k\pi}{L}\right)s\sin\frac{k\pi x}{L}+s^2\varphi'_1+s^3\varphi'_2+o(s^3),\label{33} \\
(\bar u-u)u&=&-\bar u\big(Q_ks\cos\frac{k\pi x}{L}+\psi_1s^2+\psi_2s^3+o(s^3)\big)\\ \nonumber
&&-\big(Q_ks\cos\frac{k\pi x}{L}+\psi_1s^2+\psi_2s^3+o(s^3)\big)^2, \label{34} \\ \nonumber
\end{eqnarray}
and
\begin{eqnarray}\label{35}
\Phi(u,v)&=&\Phi(\bar u,\bar v)+\Phi_u(\bar u,\bar v)u+\Phi_v(\bar u,\bar v)v \\ \nonumber
&&+\frac{1}{2}\big(\Phi_{uu}(\bar u,\bar v)u^2+\Phi_{vv}(\bar u,\bar v)v^2+2\Phi_{uv}(\bar u,\bar v)uv\big)+o((u-\bar u)^2,(v-\bar v)^2)\\ \nonumber
&=&\Phi(\bar u,\bar v)+s\Big(\Phi_u(\bar u,\bar v)Q_k\cos\frac{k\pi x}{L}+\Phi_v(\bar u,\bar v)\cos\frac{k\pi x}{L}\Big)+s^2\Big(Q_k\Phi_u(\bar u,\bar v)\psi_1\\
&&+\Phi_v(\bar u,\bar v)\varphi_1+\big(\frac{1}{2}\Phi_{uu}(\bar u,\bar v)Q_k^2+\frac{1}{2}\Phi_{vv}(\bar u,\bar v)+\Phi_{uv}(\bar u,\bar v)Q_k\big)\cos^2\frac{k\pi x}{L}\Big)+o(s^2), \nonumber
\end{eqnarray}
where again the little-$o$ terms are taken with respect to the $H^2$-norms.  After substituting the terms (\ref{32})-($\ref{35}$) into the $u$-equation of (\ref{11}), we obtain that
\begin{eqnarray} \nonumber
&&sD_1\left(\frac{k\pi}{L}\right)^2Q_k\cos\frac{k\pi x}{L}-s^2D_1\psi_1''-s^3D_1\psi_2'' \\ \nonumber
&=&-\bar u\big(sQ_k\cos\frac{k\pi x}{L}+s^2\psi_1+s^3\psi_2+o(s^3)\big)-\Big(sQ_k\cos\frac{k\pi x}{L}+s^2\psi_1+s^3\psi_2+o(s^3)\Big)^2 \label{36} \\
&&-\Big(\chi_k+\mathcal{K}_2s+\mathcal{K}_3s^2+o(s^3)\Big) \Big((P_0+P_1s+P_2s^2)\big( \chi_k+\mathcal{K}_2s+\mathcal{K}_3s^2+o(s^3)\big)'\Big)', \\ \nonumber
\end{eqnarray}
where we have used the notations that $P_0=\Phi(\bar u,\bar v)$, $P_1=\Phi_u(\bar u,\bar v)Q_k\cos\frac{k\pi x}{L}+\Phi_v(\bar u,\bar v)\cos\frac{k\pi x}{L}$,  and \[P_2=Q_k\Phi_u(\bar u,\bar v)\psi_1+\Phi_v(\bar u,\bar v)\varphi_1+\Big(\frac{1}{2}\Phi_{uu}(\bar u,\bar v)Q_k^2+\frac{1}{2}\Phi_{vv}(\bar u,\bar v)+\Phi_{uv}(\bar u,\bar v)Q_k\Big)\cos^2\frac{k\pi x}{L}.\]
Equating the $s^2$ terms in (\ref{36}), we have that
\begin{eqnarray}\label{37}
&&D_1\psi_1''-\bar u\psi_1-Q_k^2\cos^2\frac{k\pi x}{L}+\Phi(\bar u,\bar v)\left(\frac{k\pi}{L}\right)^2\mathcal{K}_2\cos\frac{k\pi x}{L}\\ \nonumber
&=&\chi_k\Big(-(Q_k\Phi_u(\bar u,\bar v)+\Phi_v(\bar u,\bar v))\left(\frac{k\pi}{L}\right)^2 \cos\frac{2k\pi x}{L}+\Phi(\bar u,\bar v)\varphi_1''\Big).\\ \nonumber
\end{eqnarray}
Multiplying (\ref{37}) by $\cos\frac{k\pi x}{L}$ and integrating it over $(0,L)$ by parts give rise to
\begin{eqnarray} \label{38}
&&\left(\frac{k\pi}{L}\right)^2\Phi(\bar u,\bar v)\mathcal{K}_2\int_0^L\cos^2\frac{k\pi x}{L}\mathrm{d}x \\ \nonumber
&=&\Big(D_1\big(\frac{k\pi}{L}\big)^2 +\bar u \Big)\int_0^L\psi_1\cos\frac{k\pi x}{L}\mathrm{d}x-\big(\frac{k\pi}{L}\big)^2  \chi_k\Phi(\bar u,\bar v)\int_0^L \varphi_1 \cos\frac{k\pi x}{L}\mathrm{d}x.\\ \nonumber
\end{eqnarray}
Substituting (\ref{31}) into the $v$-equation in (\ref{11}), we obtain that
\begin{align}\label{39}
&D_2\left(\frac{k\pi}{L}\right)^2s \cos\frac{k\pi x}{L}-s^2D_2\varphi''_1-s^3D_2\varphi_2''-o(s^3)+\Big(\bar v + s\cos\frac{k\pi x}{L}+s^2\varphi_1 +s^3\varphi_2  +o(s^3)\Big) \nonumber\\ \nonumber
=&h(\bar u)+h'(\bar u)\big(Q_ks\cos\frac{k\pi x}{L}+\psi_1s^2+\psi_2s^3+o(s^3)\big)+\frac{1}{2}h''(\bar u)\big(Q_ks\cos\frac{k\pi x}{L}+\psi_1s^2+\psi_2s^3+o(s^3)\big)^2 \nonumber \\
&+\frac{1}{6}h'''(\bar u)\big(Q_ks\cos\frac{k\pi x}{L}+\psi_1s^2+\psi_2s^3+o(s^3)\big)^3.\\ \nonumber
\end{align}
Equating the $s^2$ terms in (\ref{39}), we obtain the following equation
\begin{equation}\label{310}
D_2\varphi_1''+h'(\bar u)\psi_1-\varphi_1+\frac{1}{2}Q_k^2h''(\bar u)\cos^2\frac{k\pi x}{L}=0.
\end{equation}
Multiplying (\ref{310}) by $\cos\frac{k\pi x}{L}$ and integrating it over $(0,L)$ by parts, we have that
\begin{equation}\label{311}
h'(\bar u)\int_0^L\psi_1\cos\frac{k\pi x}{L}\mathrm{d}x-\Big(D_2\big(\frac{k\pi}{L}\big)^2+1\Big)\int_0^L\varphi_1\cos\frac{k\pi x}{L}\mathrm{d}x=0
\end{equation}
On the other hand, since $(\psi_1,\varphi_1)\in \mathcal{Z}$,  we have from (\ref{218}) that
\begin{equation}\label{312}
\int_0^L\big(Q_k\psi_1+\varphi_1\big)\cos\frac{k\pi x}{L}\mathrm{d}x=Q_k\int_0^L\psi_1\cos\frac{\varphi x}{L}\mathrm{d}x+\int_0^L\psi_1\cos\frac{k\pi x}{L}\mathrm{d}x=0
\end{equation}
From (\ref{311}) and (\ref{312}), we arrive at the following system
\begin{equation}\label{313}
\begin{pmatrix}
 h'(\bar u)   ~~~&-D_2\left(\frac{k\pi}{L}\right)^2-1 \\
 Q_k            ~~~   & 1
  \end{pmatrix} \begin{pmatrix}
   \int_0^L\psi_1\cos\frac{k\pi x}{L}\mathrm{d}x\\
  \int_0^L\varphi_1\cos\frac{k\pi x}{L}\mathrm{d}x
   \end{pmatrix}=
   \begin{pmatrix}
0 \\
0
  \end{pmatrix}.
   \end{equation}
It is easy to see that the coefficient matrix of (\ref{313}) is nonsingular, therefore we must have that
\begin{equation}\label{314}
\int_0^L\psi_1\cos\frac{k\pi x}{L}\mathrm{d}x=\int_0^L\varphi_1\cos\frac{k\pi x}{L}\mathrm{d}x=0.
\end{equation}
Putting (\ref{314}) into (\ref{38}), we readily see that $\mathcal{K}_2=0$ and hence we have proved the following observation.
\begin{proposition}\label{proposition2}
Assume that the conditions (\ref{17}), (\ref{18}) and (\ref{29}) are satisfied.  Then $\mathcal{K}_2=0$ and the local bifurcation curve of (\ref{11}) at $(\bar u,\bar v,\chi_k)$ is of pitchfork type if $\mathcal{K}_3\neq0$.
\end{proposition}

We want to mention that it is shown that the local steady state bifurcation for reaction--advection--diffusion system is pitch--fork in general when the domain is a 1D interval.  However this is not necessary true in higher space dimensions.
%As we see in Theorem \ref{theorem21}, we expand the bifurcating solution at $(\bar u,\bar v,\chi_k)$ as
%\begin{equation*}
%\left \{
%\begin{array}{ll}
%u_k(s,x)=\bar u + sQ_k\cos\frac{kk\pi x}{L}+s^2\psi_1(k) +s^3\psi_2(k)  +o(s^3),\\
%v_k(s,x)=\bar v + s\cos\frac{kk\pi x}{L}+s^2\varphi_1(k) +s^3\varphi_2(k)  +o(s^3),\\
%\chi_k(s)= \chi_k+ \mathcal{K}_2s+\mathcal{K}_3s^2+o(s^2),\\
%\end{array}
%\right.
%\end{equation*}
%then following the same calculations that lead to (\ref{49}), we can have the following identities
%\begin{equation*}
%\int_0^L\psi_1(k)\cos\frac{kk\pi x}{L}\mathrm{d}x=\int_0^L\varphi_1(k)\cos\frac{kk\pi x}{L}\mathrm{d}x=0,
%\end{equation*}
%which implies that $\mathcal{K}_2=0$ for all $k\geq1$.  Therefore, the bifurcation at $(\bar u,\bar v,\chi_k)$ is pitchfork for all $k\in\mathbb{N}^+$.

\subsection{Bifurcation direction}
Now we proceed to calculate the sign of $\mathcal{K}_3$ to determine the turning direction and the stability of $\Gamma_k(s)$ at $(\bar u,\bar v,\chi_k)$.  To this end, we need to collect the $s^3$-terms in (\ref{36}) and (\ref{39}).  For the simplicity of calculations we make the following particular choices of $\Phi$ and $h(u)$,
\begin{equation}\label{315}
\Phi(u,v)=u,~h(u)= \beta u,\beta>0,
\end{equation}
therefore we will study the stability of $(u_k(s,x),v_k(s,x))$ of the following system in the rest part of this section,
\begin{equation}\label{316}
\left\{
\begin{array}{ll}
\big(D_1u'- \chi u v'\big)'+(\bar u-u)u=0,  & x \in(0,L),\\
D_2 v'' -v + \beta u=0, & x \in(0,L),\\
u'(x)=v'(x)=0,&x=0,L.
\end{array}
\right.
\end{equation}
Since $\mathcal{K}_2=0$, we readily see that collecting the $s^2$ terms of (\ref{36}) and (\ref{39}) leads us to
\begin{equation}\label{317}
\left \{
\begin{array}{ll}
D_1\psi''_1-\chi_k \bar u \varphi''_1- \bar u \psi_1+\chi_kQ_k\left(\frac{k\pi}{L}\right)^2\cos \frac{2k\pi x}{L}-Q_k^2\cos^2\frac{k\pi x}{L}=0,&x\in(0,L),\\
D_2\varphi''_1-\varphi_1+\beta \psi_1=0,&x\in(0,L),\\
\psi'_1(x)=\varphi'_1(x)=0,&x=0,L.
\end{array}
\right.
\end{equation}
Moreover, by collecting the $s^3$ terms of (\ref{36}) and (\ref{39}), together with the assumption (\ref{315}) and the fact that $\mathcal{K}_2=0$, we arrive at the following system
\begin{equation}\label{318}
\left \{
\begin{array}{ll}
D_1\psi''_2-\bar u\psi_2-2Q_k\psi_1\cos\frac{k\pi x}{L}+ \chi_kQ_k(\frac{k\pi}{L})\varphi'_1 \sin \frac{k\pi x}{L}+ \chi_k \left(\frac{k\pi}{L}\right) \psi'_1 \sin\frac{k\pi x}{L}\\
-\chi_k \bar u\varphi''_2-\chi_k Q_k \varphi''_1 \cos \frac{k\pi x}{L}+\chi_k\psi_1\left(\frac{k\pi}{L}\right)^2 \cos \frac{k\pi x}{L}+\mathcal{K}_3\bar u\left(\frac{k\pi}{L}\right)^2\cos\frac{k\pi x}{L}=0,&x\in(0,L)\\
D_2\varphi''_2-\varphi_2+\beta \psi_2=0,&x\in(0,L),\\
\psi'_2(x)=\varphi'_2(x)=0,&x=0,L.
\end{array}
\right.
\end{equation}

Following the same arguments that lead to (\ref{314}), we can also show that
\begin{equation}\label{319}
\int_0^L \psi_2\cos \frac{k\pi x}{L}dx=\int_0^L \varphi_2\cos \frac{k\pi x}{L}dx=0.
\end{equation}
Now multiplying (\ref{318}) by $\cos\frac{k\pi x}{L}$ and integrating it over $(0,L)$ by parts, we obtain that
\begin{eqnarray}\nonumber
&&D_1\int_0^L\psi_2''\cos\frac{k\pi x}{L}\mathrm{d}x-\bar u\int_0^L\psi_2\cos\frac{k\pi x}{L}\mathrm{d}x-2Q_k\int_0^L\psi_1\cos^2\frac{k\pi x}{L}\mathrm{d}x\\ \nonumber
&&+\chi_kQ_k\left(\frac{k\pi}{L}\right)\int_0^L\varphi_1'\sin\frac{k\pi x}{L}\cos\frac{k\pi x}{L}\mathrm{d}x+ \chi_k\left(\frac{k\pi}{L}\right)\int_0^L\psi_1'\sin\frac{k\pi x}{L}\cos\frac{k\pi x}{L}\mathrm{d}x\\ \nonumber
&&- \chi_k\bar u\int_0^L\varphi_2''\cos\frac{k\pi x}{L}\mathrm{d}x- \chi_kQ_k\int_0^L\varphi_1''\cos^2\frac{k\pi x}{L}\mathrm{d}x+ \chi_k\left(\frac{k\pi}{L}\right)^2\int_0^L\psi_1\cos^2\frac{k\pi x}{L}\mathrm{d}x \label{320} \\
&&+\bar u\left(\frac{k\pi}{L}\right)^2\mathcal{K}_3\int_0^L\cos^2\frac{k\pi x}{L}=0.
\end{eqnarray}
Substituting (\ref{319}) into (\ref{320}), we have from the integration by parts that
\begin{eqnarray}\nonumber
\frac{\bar uk\pi^2}{2L}\mathcal{K}_3&=&\Big(Q_k-\frac{\chi_k}{2}\left(\frac{k\pi}{L}\right)^2\Big)\int_0^L\psi_1\mathrm{d}x+\Big(Q_k+\frac{\chi_k}{2}\left(\frac{k\pi}{L}\right)^2\Big)\int_0^L\psi_1\cos\frac{2k\pi x}{L}\mathrm{d}x \label{321} \\
&-& \chi_kQ_k\left(\frac{k\pi}{L}\right)^2\int_0^L\varphi_1\cos\frac{2k\pi x}{L}\mathrm{d}x.
\end{eqnarray}
Therefore, in order to calculate $\mathcal{K}_3$, we will need to evaluate the following integrals: \[\int_0^L\psi_1\mathrm{d}x, \int_0^L\psi_1\cos\frac{2k\pi x}{L}\mathrm{d}x, \text{~and~} \int_0^L\varphi_1\cos\frac{2k\pi x}{L}\mathrm{d}x.\]
To compute the last two integrals, we multiply the first equation (\ref{317}) by $\cos\frac{2k\pi x}{L}$ and integrate it over $(0,L)$.  Then through straightforward calculations we obtain that
\begin{align}\label{322}
   &-\Big(D_1\big(\frac{2k\pi}{L}\big)^2+\bar u\Big)\int_0^L\psi_1\cos\frac{2k\pi x}{L}\mathrm{d}x+\chi_k\bar u \Big(\frac{2k\pi}{L}\Big)^2\int_0^L\varphi_1\cos\frac{2k\pi x}{L}\mathrm{d}x  \nonumber \\
   =&-\frac{Q^2_kL}{2\bar u}\Big(D_1\big(\frac{k\pi}{L} \big)^2+\frac{\bar u}{2} \Big).
\end{align}
Multiplying the second equation in (\ref{317}) by $\cos\frac{2k\pi x}{L}$ and integrating it over $(0,L)$ by parts, we have from straightforward calculations that
\begin{equation}\label{323}
\frac{D_2\big(\frac{2k\pi}{L}\big)^2+1}{\beta} \int_0^L\varphi_1\cos\frac{2k\pi x}{L}\mathrm{d}x=Q_{2k}\int_0^L\varphi_1\cos\frac{2k\pi x}{L}\mathrm{d}x= \int_0^L\psi_1\cos\frac{2k\pi x}{L}\mathrm{d}x,
\end{equation}
where we use the notation that $Q_{2k}=\frac{D_2\big(\frac{2k\pi}{L}\big)^2+1}{\beta}$.  We see that equations (\ref{322}) and (\ref{323}) are equivalent to
\begin{equation*}
\begin{pmatrix}
-\Big(D_1\big(\frac{2k\pi}{L}\big)^2+\bar u\Big)  &\chi_k \bar u\left(\frac{2k\pi}{L}\right)^2 \\
\beta             & -\Big(D_2\big(\frac{2k\pi}{L}\big)^2+1\Big)
  \end{pmatrix} \begin{pmatrix}
   \int_0^L\psi_1\cos\frac{2k\pi x}{L}\mathrm{d}x\\
  \int_0^L\varphi_1\cos\frac{2k\pi x}{L}\mathrm{d}x
   \end{pmatrix}=\begin{pmatrix}
-\frac{Q^2_kL}{2\bar u}\big(D_1\big(\frac{k\pi}{L} \big)^2+\frac{\bar u}{2} \big) \\
0
  \end{pmatrix}.
   \end{equation*}
We note that this system is solvable thanks to (\ref{29}) (which implies that $\chi_{k}\neq\chi_{2k}$) since bifurcation occurs at $(\bar u,\bar v,\chi_k)$.  Moreover we can have from straightforward calculations that,
\begin{equation}\label{324}
\int_0^L\psi_1\cos\frac{2k\pi x}{L}\mathrm{d}x=\frac{\big(D_2\left(\frac{2k\pi}{L}\right)^2+1\big) \frac{Q^2_KL}{2\bar u}\big(D_1\big(\frac{k\pi}{L} \big)^2+\frac{\bar u}{2} \big)}{12D_1D_2\left(\frac{k\pi}{L}\right)^4-3\bar u};
\end{equation}
\begin{equation}\label{325}
\int_0^L\varphi_1\cos\frac{2k\pi x}{L}\mathrm{d}x=\frac{\beta\frac{Q^2_KL}{2\bar u}\big(D_1\big(\frac{k\pi}{L} \big)^2+\frac{\bar u}{2} \big)}{12D_1D_2\left(\frac{k\pi}{L}\right)^4-3\bar u};
\end{equation}
on the other hand, integrating the first equation in (\ref{317}) over $(0,L)$ by parts leads us to
\begin{equation}\label{326}
\int_0^L\psi_1\mathrm{d}x=-\frac{Q_k^2L}{2\bar u}.
\end{equation}
By putting (\ref{323})-(\ref{326}) together, we conclude from (\ref{321}) and straightforward calculations that
\begin{equation}\label{327}
\frac{\bar uk\pi^2}{2L}\mathcal{K}_3=\frac{Q_k^3L}{16D_2(\frac{k\pi}{L})^4}\frac{F(D_1)}{D_1-\frac{\bar u}{4D_2}\left(\frac{L}{k\pi}\right)^4}=\frac{Q_k^3L}{16D_2(\frac{k\pi}{L})^4}\frac{aD_1^2+bD_1+c}{D_1-\frac{\bar u}{4D_2}\left(\frac{L}{k\pi}\right)^4},
\end{equation}
where \[a=\frac{14D_2(\frac{k\pi}{L})^6-(\frac{k\pi}{L})^4}{\bar u^2},~b=-\frac{2D_2(\frac{k\pi}{L})^4+5(\frac{k\pi}{L})^2}{2\bar u},~c=5D_2\Big(\frac{k\pi}{L}\Big)^2+\frac{7}{2}.\]
Moreover, if $D_2\in\Big(0,\frac{113}{1116}\big(\frac{k\pi}{L}\big)^2\Big)$, the quadratic equation $F(r)=0$ has two roots
\[r_1=\frac{\bar u\Big(2D_2\big(\frac{k\pi}{L}\big)^2+5-\sqrt{-1116D_2^2\big(\frac{k\pi}{L}\big)^4-684D_2\big(\frac{k\pi}{L}\big)^4+81} \Big)}{14D_2\big(\frac{k\pi}{L}\big)^4-\big(\frac{k\pi}{L}\big)^2}\]
and \[r_2=\frac{\bar u\Big(2D_2\big(\frac{k\pi}{L}\big)^2+5+\sqrt{-1116D_2^2\big(\frac{k\pi}{L}\big)^4-684D_2\big(\frac{k\pi}{L}\big)^4+81} \Big)}{14D_2\big(\frac{k\pi}{L}\big)^4-\big(\frac{k\pi}{L}\big)^2};\]
In particular, if $D_2=\frac{1}{14}\big(\frac{k\pi}{L}\big)^2$, we have that $r_1=r_2=\frac{c}{b}$.  Now by denoting
\[r_3=\frac{\bar u}{4D_2}\big(\frac{L}{k\pi}\big)^4,\]
we present the following results that characterize the sign of $\mathcal{K}_3$ hence the turning direction of the bifurcation branch $\Gamma_k$ around $(\bar u,\bar v,\chi_k)$.
\begin{theorem}\label{theorem31}
The bifurcation curve $\Gamma_k(s)$ of (\ref{316}) around $(\bar u,\bar v,\chi_k)$ turns to the right if $\mathcal{K}_3>0$ and to the left if $\mathcal{K}_3<0$.  Moreover, we have the following cases:

(i)  when $D_2\in(0,\frac{1}{14}\left(\frac{L}{k\pi}\right)^2]$, we have that $\mathcal{K}_3>0$ if $D_1\in(r_2,r_3)$ and $\mathcal{K}_3<0$ if $D_1\in(0,r_2)\cup(r_3,\infty)$;

(ii) when $D_2\in(\frac{1}{14}\left(\frac{L}{k\pi}\right)^2,\frac{1}{10}\left(\frac{L}{k\pi}\right)^2)$, we have that $r_1<r_3<r_2$ and $\mathcal{K}_3>0$ if $D_1\in(r_1,r_3)\cup(r_2,\infty)$ and $\mathcal{K}_3<0$ if $D_1\in(0,r_1)\cup(r_3,r_2)$;

(iii) when $D_2=\frac{1}{10}\left(\frac{L}{k\pi}\right)^2$, we have that $r_1=r_3$ and $\mathcal{K}_3>0$ if $D_1\in(r_2,\infty)$ and $\mathcal{K}_3<0$ if $D_1\in(0,r_1)\cup(r_1,r_2)$;

(iv) when $D_2\in(\frac{1}{10}\left(\frac{L}{k\pi}\right)^2,\frac{113}{1116}\left(\frac{L}{k\pi}\right)^2)$, we have that $r_3<r_1<r_2$ and $\mathcal{K}_3>0$ if $D_1\in(r_3,r_1)\cup(r_2,\infty)$ and $\mathcal{K}_3<0$ if $D_1\in(0,r_3)\cup(r_1,r_2)$;

(v) when $D_2=\frac{113}{1116}\left(\frac{L}{k\pi}\right)^2$, we have that $r_1=r_2$ and $\mathcal{K}_3>0$ if $D_1\in(r_3,r_1)\cup(r_1,\infty)$ and $\mathcal{K}_3<0$ if $D_1\in(0,r_3)$;

(vi) when $D_2\in(\frac{113}{1116}\left(\frac{L}{k\pi}\right)^2,\infty)$, we have that $\mathcal{K}_3>0$ if $D_1\in(r_3,\infty)$ and $\mathcal{K}_3<0$ if $D_1\in(0,r_3)$.
\end{theorem}

\begin{proof} We observe from (\ref{327}) that $\mathcal{K}_3$ has the same sign as  \[\frac{F(D_1)}{D_1-\frac{\bar u}{4D_2}\left(\frac{L}{k\pi}\right)^4}=\frac{aD_1^2+bD_1+c}{D_1-\frac{\bar u}{4D_2}\left(\frac{L}{k\pi}\right)^4}.\]We divide our discussions into the following cases.

If $D_2\leq\frac{1}{14}\big(\frac{L}{k\pi}\big)^2$, we see that $a\leq0$ and $F(0)>0$, therefore $F(D_1)=0$ always has two roots $r_1$ and $r_2$ with $r_1<0<r_2$ if $a<0$ and one root $r_2>0$ if $a=0$.  On the other hand, we can have from straightforward calculations that \[F(r_3)=\frac{1}{D_2(\frac{k\pi}{L})^2}\Big(\frac{7}{8}-\frac{1}{16D_2(\frac{k\pi}{L})^2}\Big)-\Big(\frac{1}{4}+\frac{5}{8D_2(\frac{k\pi}{L})^2}\Big)+5D_2\left(\frac{k\pi}{L}\right)^2+\frac{7}{2}=0\]
has two negative roots $D_2=-\frac{1}{2}\big(\frac{L}{k\pi}\big)^2$, $D_2=-\frac{1}{4}\big(\frac{L}{k\pi}\big)^2$, and one positive root $D_2=\frac{1}{10}\big(\frac{L}{k\pi}\big)^2$.  Moreover, $F(r_3)$ is monotone increasing in $D_2$ for all $D_2>0$, $F(r_3)<0$ if $D_2<\frac{1}{10}\big(\frac{L}{k\pi}\big)^2$ and $F(r_3)>0$ if $D_2>\frac{1}{10}\big(\frac{L}{k\pi}\big)^2$.  Hence we have that $0<r_3<r_2$ in this case and the conclusions in case \emph{(i)} hold.

When $D_2>\frac{1}{14}\big(\frac{L}{k\pi}\big)^2$, we have that $a>0$ and $F(0)>0$.  On the other hand, we see the determinant of the quadratic equation $F(D_1)=0$ in (\ref{327}) is  \[\bigtriangleup=\frac{-1116D_2^2(\frac{k\pi}{L})^8-684D_2(\frac{k\pi}{L})^6+81(\frac{k\pi}{L})^4}{4 \bar u^2},\]
and it follows from straightforward calculations that $\bigtriangleup >0$ if $D_2\in(0,\frac{113}{1116}\left(\frac{L}{k\pi}\right)^2)$ and $\bigtriangleup <0$ if $D_2\in(\frac{113}{1116}\left(\frac{L}{k\pi}\right)^2,\infty)$.  Recall that $F(r_3)$ has a unique positive root $D_2=\frac{1}{10}\big(\frac{L}{k\pi}\big)^2$ and now we continue our analysis in the following subcases:

If $D_2\in(\frac{1}{14}\big(\frac{L}{k\pi}\big)^2,\frac{113}{1116}\big(\frac{L}{k\pi}\big)^2)$, then we have that $F(D_1)=0$ has two positive roots $r_1$ and $r_2$ with $r_1<r_3<r_2$, since $F(r_3)<0$ in this case.  Therefore we arrive at the conclusions in case \emph{(ii)}.

If $D_2=\frac{1}{10}\big(\frac{L}{k\pi}\big)^2$, we have that $F(r_3)=0$ and in particular we have that $r_3=r_1$.  Thus $\mathcal{K}_3$ is a straight line as a function of $D_1$, which implies case \emph{(iii)}.

If $D_2\in(\frac{1}{10}\big(\frac{L}{k\pi}\big)^2,\frac{113}{1116}\big(\frac{L}{k\pi}\big)^2)$, we have that $F(r_3)>0$.  Following the same arguments in case 2, we can show case \emph{(iv)} since one has in this case that $r_3<r_1<r_2$.  Moreover, if $D_2=\frac{113}{1116}\big(\frac{L}{k\pi}\big)^2$, we have that $r_3<r_1=r_2$, and this implies the statements in case \emph{(v)}.

Finally, if $D_2\in(\frac{113}{1116}\big(\frac{L}{k\pi}\big)^2,\infty)$, then the determinant of $F(D_1)$ is negative and $F(D_1)>0$ for all $D_1>0$.  Therefore, we have that $\mathcal{K}_3>0$ if $D_1>r_3$ and $\mathcal{K}_3<$ if $D_1<r_3$.  This finishes the proof of Theorem \ref{theorem31}.
\end{proof}
%%
%\begin{figure}
%\minipage{0.32\textwidth}
%  \includegraphics[width=\linewidth]{1.PNG}
%  \caption*{$D_2\in(0,\frac{1}{14}(\frac{L}{k\pi})^2]$}
%\endminipage\hfill
%\minipage{0.32\textwidth}
%  \includegraphics[width=\linewidth]{2.PNG}
%  \caption*{$D_2\in(\frac{1}{14}(\frac{L}{k\pi})^2,\frac{1}{10}(\frac{L}{k\pi})^2)$}
%\endminipage\hfill
%\minipage{0.32\textwidth}%
%  \includegraphics[width=\linewidth]{3.PNG}
%  \caption*{$D_2=\frac{1}{10}(\frac{L}{k\pi})^2$}\label{ }
%\endminipage\\
%\minipage{0.32\textwidth}
%  \includegraphics[width=\linewidth]{4.PNG}
%  \caption*{$D_2\in(\frac{1}{10}(\frac{L}{k\pi})^2,\frac{113}{1116}(\frac{L}{k\pi})^2)$}\label{ }
%\endminipage\hfill
%\minipage{0.32\textwidth}
%  \includegraphics[width=\linewidth]{5.PNG}
%  \caption*{$D_2=\frac{113}{1116}(\frac{L}{k\pi})^2$}\label{fig:awesome_image2}
%\endminipage\hfill
%\minipage{0.32\textwidth}%
%  \includegraphics[width=\linewidth]{6.PNG}
%  \caption*{$D_2\in(\frac{113}{1116}(\frac{L}{k\pi})^2,\infty)$}\label{fig:awesome_image3}
%\endminipage
%\end{figure}

The graphes of $\mathcal{K}_3$ in Theorem \ref{theorem31} are illustrated in figures (1)-(6).  As we can see in Theorem \ref{theorem31}, $\mathcal{K}_3$ has a singularity at $D_1=\frac{\bar u}{4D_2}\Big(\frac{L}{k\pi}\Big)^4$.  However, we already show in (\ref{29}), bifurcation occurs at $(\bar u,\bar v,\chi_k)$ only if $D_1\neq\frac{\bar u}{4D_2}\Big(\frac{L}{k\pi}\Big)^4$.  Actually, if $D_1=\frac{\bar u}{4D_2}\Big(\frac{L}{k\pi}\Big)^4$, then we have that $\mathcal{K}_3=\infty$ and formally the curve $\Gamma(s)$ must coincide with the $\chi$-axis.  However we do not consider this singularity case in this paper.

\subsection{Stability analysis}
To study the stability of $(u_k(s,x),v_k(s,x),\chi_k(s))$ around $(\bar u ,\bar v,\chi_k)$, we linearize (\ref{316}) around this bifurcating solution.  According to the principle of the linearized stability e.g. Theorem 8.6 in \cite{CR2}, to show that they are asymptotically stable, we need to prove that the each eigenvalue $\lambda$ of the following elliptic problem has negative real part:
\[D_{(u,v)}\mathcal{F}(u_k(s,x),v_k(s,x),\chi_k(s))(u,v)=\lambda (u,v),~(u,v)\in\mathcal{X} \times \mathcal{X}.\]
We readily see that this eigenvalue problem is equivalent to
\begin{equation}\label{328}
\left\{
\begin{array}{ll}
  D_1 \Big(u' - \chi_k(s)\big(uv'_1(s,x)+u_k(s,x)v'\big) \Big)' + (\bar u-2u_k(s,x))u=\lambda u,&x\in(0,L), \\
  D_2 v'' -v_k(s,x)+\beta u_k(s,x)=\lambda v,&x\in(0,L),\\
u'(x)=v'(x)=0,&x=0,L,
\end{array}
\right.
\end{equation}
where $u_k(s,x)$, $v_k(s,x)$ and $\chi_k(s)$ are as defined in Theorem \ref{theorem21}.

We first observe that $0$ is a simple eigenvalue of $D_{(u,v)}\mathcal{F}(\bar u,\bar v,\chi_k)$ with an eigen-space equal to $\text{span}\{(Q_k\cos \frac{k\pi x}{L},\cos \frac{k\pi x}{L})\}$.  In the coming analysis we assume that \[\min_{k\in\mathbb N^+}\chi_k=\chi_{k_0}.\]
For each $k\neq k_0$, we already know from Proposition \ref{proposition1} that when $s=0$, (\ref{328}) has an eigenvalue with positive real part, then from the standard eigenvalue perturbation theory e.g in \cite{Kato}, it always has a positive root for small $s$.  This implies that the bifurcation branch $\Gamma_k(s)$ around $(\bar u,\bar v,\chi_k)$ is unstable for each $k\in\mathbb N^+\backslash\{k_0\}$.

In the following we are left to investigate the stability of $\Gamma_{k_0}(s)$ for $|s|$ being small.  It follows from Corollary 1.13 in \cite{CR2} that, there exist an internal $I$ with $\chi_{k_0}\in I$ and continuously differentiable functions $\chi\in I\rightarrow \mu(\chi),~s\in(-\delta,\delta) \rightarrow \lambda(s)$ with $\lambda(0)=0$ and $\mu(\chi_{k_0})=0$ such that, $\lambda(s)$ is an eigenvalue of (\ref{328}) and $\mu(\chi)$ is an eigenvalue of the following eigenvalue problem
\begin{equation}\label{329}
D_{(u,v)}\mathcal{F}(\bar u,\bar v,\chi)(u,v)=\mu(u,v),~(u,v)\in \mathcal{X} \times \mathcal{X};
\end{equation}
moreover, $\lambda(s)$ is the only eigenvalue of (\ref{328}) in any fixed neighbourhood of the origin of the complex plane (the same assertion can be made on $\mu(\chi)$).  We also know from \cite{CR2} that the eigenfunctions of (\ref{329}) can be represented by $(u(\chi,x),v(\chi,x))$ which depend on $\chi$ smoothly and are uniquely determined through $\big(u(\chi_{k_0},x),v(\chi_{k_0},x)\big)=\big(Q_k\cos \frac{{k_0}\pi x}{L},\cos \frac{{k_0}\pi x}{L} \big)$ together with $\big(u(\chi,x)-Q_k\cos \frac{{k_0}\pi x}{L},v(\chi,x)-\cos \frac{{k_0}\pi x}{L} \big)  \in \mathcal{Z}$.

Now we are ready to present another main result of our paper.
\begin{theorem}\label{theorem32}
For $s\in(-\delta,\delta)$, $s\neq0$, the solution $(u_{k_0}(s,x),v_{k_0}(s,x))$ of (\ref{316}) is asymptotically stable if $\mathcal{K}_3>0$ and it is unstable if $\mathcal{K}_3<0$.  For each $k\in\mathbb N^+\backslash\{k_0\}$, the bifurcating solutions $(u_{k}(s,x),v_{k}(s,x))$, $s\in(-\delta,\delta)$ are always unstable.
\end{theorem}

\begin{proof}
We differentiate (\ref{329}) with respect to $\chi$ and set $\chi=\chi_{k_0}$, then since $\mu(\chi_{k_0})=0$, we arrive at the following system
\begin{equation}\label{330}
\left\{
\begin{array}{ll}
  D_1 \dot{u}''-\bar u v_1''-\chi_{k_0}\bar u\dot{v}''-\bar u\dot{u}=\dot{\mu}(\chi_{k_0})u_1,&x\in(0,L), \\
  D_2 \dot{v}'' -\dot{v}+\beta \dot{u}=\dot{\mu}(\chi_{k_0})v_1,&x\in(0,L),\\
\dot{u}'(x)=\dot{v}'(x)=0,&x=0,L,
\end{array}
\right.
\end{equation}
where $(u_k,v_k)=(Q_k \cos \frac{{k_0}\pi x}{L},\cos \frac{{k_0}\pi x}{L})$.  The dot-sign means the differentiation with respect to $\chi$ evaluated at $\chi=\chi_{k_0}$ and in particular $\dot{u}=\frac{\partial u(\chi,x)}{\partial \chi}\big\vert_{\chi=\chi_{k_0}}$, $\dot{v}=\frac{\partial v(\chi,x)}{\partial \chi}\big\vert_{\chi=\chi_{k_0}}$.

Multiplying both equations of (\ref{330}) by $\cos\frac{{k_0}\pi x}{L}$ and then integrating over $(0,L)$ by parts, we obtain that
\[\begin{pmatrix}
-D_1 \big(\frac{{k_0}\pi}{L}\big)^2-\bar u & \chi_{k_0}\bar u \big(\frac{ {k_0}\pi}{L}\big)^2\\
\beta & -D_2\big(\frac{{k_0}\pi}{L}\big)^2-1
\end{pmatrix}
\begin{pmatrix}
\int_0^L \dot{u}\cos\frac{{k_0}\pi x}{L}dx \\
~~\\
\int_0^L \dot{v}\cos\frac{{k_0}\pi x}{L}dx
\end{pmatrix}=\begin{pmatrix}
\Big(\dot{\mu}(\chi_{k_0})Q_{k_0}-\bar u\big(\frac{{k_0}\pi}{L} \big)^2\Big) \frac{L}{2}\\
~~\\
\dot{\mu}(\chi_{k_0})\frac{L}{2}
\end{pmatrix}.
\]
We know that the coefficient matrix is singular, hence in order for the system above to be solvable, we must have that
\[\dot{\mu}(\chi_{k_0})=\frac{\beta\bar u(\frac{{k_0}\pi}{L})^2}{D_1(\frac{{k_0}\pi}{L})^2+\beta Q_{k_0}+\bar u},\]
which is strictly positive.  By Theorem 1.16 in \cite{CR2}, for $s\in(-\delta,\delta)$, the functions $\lambda(s)$ and $-s\chi'_{k_0}(s)\dot{\mu}(\chi_{k_0})$ have the same zeros and the same signs.  Moreover
\[\lim_{s\rightarrow 0,~\lambda(s)\neq0}\frac{-s\chi'_{k_0}(s)\dot{\mu}(\chi_{k_0})}{\lambda(s)}=1.\]
Now, since $\mathcal{K}_2=0$, it follows that $\lim_{s\rightarrow 0} \frac{s^2\mathcal{K}_3 \dot{\mu}(\chi_{k_0})}{\lambda(s)}=-1$ and we readily see that $\text{sgn}(\lambda(s))=\text{sgn}(-\mathcal{K}_3)$ for $s\in(-\delta,\delta)$, $s\neq0$.  Therefore we have proved Theorem \ref{theorem32}.
\end{proof}

We conclude from Theorem \ref{theorem31} and Theorem \ref{theorem32} that, if $D_1$ is small, the small amplitude bifurcating solution $(u_{k_0}(s,x),v_{k_0}(s,x))$ is unstable for all $D_2>0$.  If $D_1$ is large, $(u_{k_0}(s,x),v_{k_0}(s,x))$ is unstable for $D_2<\frac{1}{14}(\frac{L}{k\pi})^2$ and is stable for $D_2>\frac{1}{14}(\frac{L}{k\pi})^2$.  Therefore the smallness of one of the diffusion rates $D_1$ and $D_2$ is sufficient to inhibit the stability of this small amplitude solution and we may expect solutions of large amplitude in this case as we shall see in the next section.

\section{Asymptotic behavior of positive monotone solutions}\label{section4}
We consider the following system
\begin{equation}\label{41}
\left\{
\begin{array}{ll}
\big(D_1u'- \chi \Phi(u,v) v'\big)'+(\bar u-u)u=0,  & x \in(0,L),\\
D_2 v'' -v + h(u)=0, & x \in(0,L),\\
u'(x)<0,v'(x)<0,& x \in(0,L),\\
u'(x)=v'(x)=0,&x=0,L.
\end{array}
\right.
\end{equation}
and the main purpose of this section is to study the asymptotic behavior of positive solutions to (\ref{41}) as $\chi/D_1$ approaches to infinity.  In contrast to the small amplitude solutions obtained in Theorem \ref{theorem21}, we are interested in studying the existence of nonconstant positive solutions to (\ref{41}) that have large amplitudes.  The last set of our main results can be summarized as follows.
\begin{theorem}\label{theorem41}
Assume that the conditions (\ref{16})-(\ref{19}) are satisfied.  Let $(u_i,v_i)$ be a positive solution of (\ref{41}) with $(D_1,\chi)=(D_{1,i},\chi_i)$.  Then we have that
\begin{equation}\label{42}
\lim_{i\rightarrow \infty}\int_0^L u_i(x)dx \leq \bar u L;
\end{equation}
moreover, the following conclusions hold, after passing to a subsequence if necessary:

\emph{(i)}  Assume that $ \frac{\chi_i}{D_{1,i}}\rightarrow \infty$ as $\chi_i\rightarrow \infty$.  Then we have that, either $u_i \rightarrow \bar u_\infty$ locally uniformly in $(0,L]$ and $v_i \rightarrow h(\bar u_\infty)$ in $C^1([0,L])$, where $\bar u_\infty\leq \bar u$ is a positive constant or $u_i \rightarrow 0$ locally uniformly in $(0,L]$ and $v_i \rightarrow 0$ in $C^1([0,L])$; $u_i(0) \rightarrow u_\infty(0)\geq\bar u$ in both cases.

\emph{(ii)}  Assume that $ \frac{\chi_i}{D_{1,i}}\rightarrow a\in[0,\infty)$ as $\chi_i\rightarrow \infty$ or $D_{1,i} \rightarrow \infty$ (so $\chi_\infty$ is comparably large).  Then we have that $u_i\rightarrow u_\infty$ in $C^1([0,L])$ and $v_i\rightarrow v_\infty$ in $C^2([0,L])$, where either $(u_\infty,v_\infty)\equiv(\bar u,h(\bar u))$ or it is a nonconstant positive solution of the following system
\begin{equation}\label{43}
\left\{
\begin{array}{ll}
u'_\infty-a\Phi(u_\infty,v_\infty)v'_\infty=0, & x \in(0,L),\\
D_2 v''_\infty-v_\infty+h(u_\infty)=0,& x \in(0,L),\\
u'_\infty(x)=v'_\infty(x)=0,&x=0,L;
\end{array}
\right.
\end{equation}
moreover $u_\infty(x)>0$ in $[0,L)$ and $v_\infty(x)>0$ in $[0,L]$.
\end{theorem}
In case \emph{(i)}, i.e., when the chemotaxis effect is greatly stronger than the cell motility, if the first alternative occurs, cell population density and chemical concentration approach to a homogeneous station which is below the environment carrying capacity. If the second alternative in case \emph{(i)} occurs, $u_\infty$ concentrates at $x=0$ and $u_i$ takes the form of a boundary spike at $x=0$ for $\chi_i$ being large.  Though a $\delta$--type aggregation/singularity is possible, the total population of cell shrinks to zero and the chemical is fully consumed.  In case \emph{(ii)}, i.e., when the chemotaxis rate and cell motility are comparably strong, we have that the total cell population (hence chemical concentration) is positive in both alternatives; moreover the cell density matches the environment carrying capacity if it approaches to the homogeneous station in the first alternative, while one expects spatial patterns of cell density and chemical concentration in the second alternative.  Our results suggest that in the limit of large chemotaxis attraction, cellular motility tends to supports cell proliferation of chemotaxis models with logistic kinetics.  It is therefore interesting and biologically important to find or characterize optimal chemotaxis rate and/or cell motility that maximizes the total cell population.  We refer to \cite{Lou1,Lou2} for the discussion on some related population dynamics models.
\begin{remark}
There are extensive works available in literature investigating (\ref{43}) and its time-dependent multi-dimensional counterparts.  For example, if $\Phi(u,v)=u$ and $h(u)=\beta u$ for a constant $\beta>0$, Biler \cite{B} established the existence of nonconstant radially symmetric solutions of (\ref{43}) over domain $\Omega$ in $\mathbb{R}^N$, $N\geq1$.   For $\Phi(u,v)=\frac{u}{v}$ and $h(u)=\beta u$, we have that $u_\infty=Cv^a$ for some positive constant $C$.  It then follows from the classical results of Lin \emph{et al.} \cite{LNT} and Ni-Takagi \cite{NT,NT2} that, for $D_2$ being sufficient small and $a\in(1,\infty)$, (\ref{43}) admits nonconstant positive solutions with $v_\infty$ concentrating at $x=0$, which also has the form of a boundary spike.  Global existence and boundary spike solution on a plat form of (\ref{43}) in multi-dimension with $\Phi(u,v)=\frac{u}{v+c}$, $c$ being a positive constants are investigated in \cite{Wq1,Wq2}.  The analysis of (\ref{43}) with general $\Phi$ and $h$ is a delicate problem which is out of scope of this paper.  We also want to mention that the time--dependent system of (\ref{11}) has quite rich spatial-temporal dynamics as demonstrated by the numerical studies in \cite{EIM,PH}.  An alternative way to establish nontrivial solutions to (\ref{41}) is to study its shadow system.  See \cite{KT,TKMI,WGY} for example.
\end{remark}
Before proving Theorem \ref{theorem41}, we first give the following observation.
\begin{lemma}\label{lemma42}
Let $(u_i,v_i)$ be a positive solution of (\ref{41}).  Then for any $x\in(0,L]$, $\lim\sup_{i\rightarrow \infty}u_i(x)<\infty$.
\end{lemma}
\begin{proof}
We argue by contradiction and assume that there exists $x_0\in(0,L]$ and a sequence $i\rightarrow \infty$ such that $\lim_{i\rightarrow \infty}u_i(x_0)=\infty$.  Then $u_i(x) \rightarrow \infty$ for all $x\in[0,x_0]$ since $u_i$ is monotone deceasing.  By integrating the $v$-equation in (\ref{41}) over $(0,L)$, we have that
\[\int_0^L v_i(x)dx=\int_0^L h(u_i(x))dx\geq \int_0^{x_0} h(u_i(x))dx \rightarrow \infty,\]
which is a contraction to the uniform boundedness of $\Vert v_i \Vert_{H^2}$ in (\ref{214}).
\end{proof}
\begin{proof}[Proof\nopunct] \emph{of (i) of Theorem} \ref{theorem41}:  First of all, we readily see that (\ref{42}) follows from (\ref{213}).  By Lemma \ref{lemma22}, the monotonicity of $u_i$ and Helly's theorem, after passing to a subsequence as $i \rightarrow \infty$, there exists a function $u_\infty$ which is nonnegative and nonincreasing  on $[0,L]$ such that $u_i(x) \rightarrow u_\infty(x) \text{ pointwise on }(0,L]$.  Moreover, thanks to (\ref{214}) and the compact embedding $H^2(0,L) \subset \subset C^1([0,L])$, we have that, after passing to yet another subsequence as $i\rightarrow \infty$, $v_i(x) \rightarrow v_\infty(x) \text{ in } C^1([0,L])$, where $v_\infty$ is also nonincreasing on $[0,L]$.

We integrate the $u$-equation in (\ref{41}) over $(0,L)$ and have from Lemma \ref{lemma22} and Fatou's lemma that
\begin{equation}\label{44}
\int_0^L u_\infty\leq \lim_{i\rightarrow \infty}\inf\int_0^L u_i\leq \bar uL.
\end{equation}
On the other hand, we integrate the $u$-equation over the interval $(x,L)$ and have that
\begin{equation}\label{45}
D_{1,i}u'_i(x)-\chi_i\Phi(u_i,v_i)v'_i(x)=F_i(x),
\end{equation}
where $F_i(x)=\int_x^L(\bar u-u_i)u_i$.  Integrating (\ref{45}) from $x$ to $L$, dividing it by $\chi_i$ and then sending $i\rightarrow \infty$, we obtain from Lemma \ref{lemma22} that $\int_x^L\Phi(u_\infty(y),v_\infty(y))v'_\infty(y) dy=0$ for any $x$ in $(0,L]$, which implies that
\begin{equation}\label{46}
\Phi(u_\infty,v_\infty)v'_\infty(x)=0,~\text{a.e.}~x\in[0,L];
\end{equation}
Moreover, we can show that $v_\infty $ satisfies
\begin{equation}\label{47}
\left\{
\begin{array}{ll}
D_2v''_\infty-v_\infty+h(u_\infty)=0, & x \in(0,L),\\
v'_\infty(x)=v'_\infty(x)=0,&x=0,L.
\end{array}
\right.
\end{equation}

We now divide our discussions into the following two cases.  If $u_\infty\equiv \bar u_\infty$ is a positive constant, we must have from (\ref{47}) that $v_\infty\equiv  h(\bar u_\infty)$.  Moreover we have from (\ref{44}) that $\bar u_\infty\leq \bar u$.  If $u_\infty\not \equiv \bar u_\infty>0$ for $x\in(0,L]$, we claim that $u_i\rightarrow 0$ in $(0,L]$.  We argue by contradiction and suppose that there exists $x_0\in(0,L)$ such that $u_\infty>0$ for $x\in(0,x_0)$ and $u_\infty\equiv 0$ for $x\in(x_0,L]$.  There are two possibilities to consider: (a). $x_0=L$.  In this case $u_\infty(x)>0$ in $[0,L]$, which implies through (\ref{46}) that $v'_\infty(x)\equiv0$ hence $v_\infty$ is a constant, therefore according to (\ref{47}) $u_\infty$ must also be a constant, denoted by $\bar u_\infty$, however this is impossible according to our assumption that $u_\infty \not\equiv \bar u_\infty$ unless $\bar u_\infty=0$ which is just what we want to prove; (b). $x_0\in(0,L)$.  In this case, we have from (\ref{17}) and (\ref{46}) that $v'_\infty\equiv0$ in $[0,x_0^{-}]$.  Note that $v\in C^1([0,L])$.  Therefore we have $v_\infty\equiv$ \emph{one positive constant} in $[0,x_0^{-}]$ hence $u_\infty\equiv$ \emph{another positive constant} in $(0,x_0)$ in light of (\ref{47}), otherwise $u_\infty\equiv 0$ in $(0,x_0]$, hence in $(0,L]$ since $u_\infty$ is nonincreasing in $[0,L]$, therefore our claim is proved.  On the other hand, we have that \begin{equation}\label{48}
\left\{
\begin{array}{ll}
D_2v''_\infty-v_\infty=0, & x \in(x_0,L) ,\\
v'_\infty(L)=0.
\end{array}
\right.
\end{equation}
Therefore $v''_\infty(x)>0$ in $(x_0,L)$ and we have that $\lim_{x\rightarrow x_0^{+}}v'_\infty(x)<0$, however this is impossible since $\lim_{x\rightarrow x_0^{-}}v'_\infty(x)=0$ and $v_\infty(x) \in C^1([0,L])$.  Therefore we must have that $u_i\rightarrow0$ in $(0,L]$.  Moreover, it is easy to see from Lemma \ref{lemma22} that $u_i(0)>\bar u$ hence $u_\infty(0)\geq\bar u$.

\emph{Proof of (ii) of Theorem} \ref{theorem41}:
In light of (\ref{215}) in Lemma \ref{lemma23} and Lemma \ref{lemma42}, we see that as $\frac{\chi_i}{D_{1,i}} \rightarrow a\in[0,\infty)$, $u''_i$ and $v''_i$ are uniformly bounded for all $i$.  By Azela-Ascoli theorem, $u_i \rightarrow u_\infty$ in $C^1([0,L])$ as $i\rightarrow \infty$, after passing to a subsequence.  Then we conclude from (\ref{45}) that
\[u'_i(x)-\frac{\chi_i}{D_{1,i}}\Phi(u_i,v_i)v'_i(x)=\frac{1}{D_{1,i}}F_i(x).\]
Sending $i$ to $\infty$, we readily have that
\[u'_\infty-a\Phi(u_\infty,v_\infty)v'_\infty=0, \forall x\in(0,L).\]
Similarly we can show that $v_\infty$ satisfies (\ref{47}).

On the other hand, we integrate the $u_i$ equation over $(0,L)$ and have that $\int_0^Lu_i(\bar u-u_i)=0$.  Applying the Lebesgue's dominated convergence theorem we have that
\[\int_0^Lu_\infty(\bar u-u_\infty)=0.\]
Remind that $u_i(0)>\bar u$ hence $u_\infty(0)\geq \bar u$.  Therefore if $u_\infty$ is a constant it must be $\bar u$ hence $\bar v$ equals $h(\bar u)$ according to (\ref{47}).  If $u_\infty$ and $v_\infty$) are not constants, it is easy to see that they satisfy (\ref{43}).  To show that $u_\infty(x)$ and $v_\infty(x)$ are strictly positive on $[0,L)$ and $[0,L]$ respectively, we argue by contradiction.  If $v_\infty(L)=0$, we reach a contradiction to Hopf's boundary lemma in (\ref{47}), unless $v_\infty\equiv0$ which is impossible; if $v_\infty(x)=0$ in $[0,x_0)$ for some $x_0\in (0,L]$, we have that $v'_\infty(x_0)=0$ and again we reach a contradiction.  Therefore $v_\infty(x)>0$ for $x\in[0,L]$.  To show $u_\infty(x)>0$ in $[0,L)$, we suppose that there exists $x_0\in[0,L)$ such that $u_\infty\equiv0$ in $(x_0,L]$, then we must have from (\ref{47}) that $v_\infty\equiv 0$ in $(x_0,L)$ which is impossible.  Therefore \emph{(ii)} is proved.
\end{proof}
\begin{remark}
In both case (i) and case (ii), we know from that Fatou's Lemma and (\ref{42}) that the limit of total cell population must be finite in the limit of large chemotaxis attraction and/or cell motility.  Apparently this is due to the presence of logistic kinetic term from which (\ref{42}) applies.  In case (i), when $u_i$ converges to a boundary spike at $x=0$, we know that $u_\infty(0)>\bar u$ thanks to Lemma \ref{lemma22} and it is unknown whether or not it can be $\infty$.  Therefore a $\delta$-type aggregation/singularity at $x=0$ is possible in this case.  However, in case (ii), we have that $u_\infty$ is always bounded in $[0,L]$ in virtue of Lemma \ref{lemma22}, hence a $\delta$-type aggregation is impossible in this case.
\end{remark}

\section{Conclusion and discussion}\label{section5}
In this paper, we establish the sufficient condition $\chi>\min_{k\in\mathbb{N}^+}\chi_k$ for the existence of nonconstant positive solutions of (\ref{11}).  This condition is the same as that when the constant solution $(\bar u,\bar v)$ of (\ref{11}) loses its stability to nonconstant positive solutions.  See Proposition \ref{proposition1}.  We carry out global analysis of local bifurcation branches and show that they always stay within the first quadrant of $(\mathcal{X}\times\mathcal{X})\times \mathbb{R}$ and all non-compact continuum can extend to infinity only in the positive direction of the $\chi$-axis.  Stability of the bifurcating solutions around $(\bar u,\bar u,\chi_k)$ has also been investigated rigorously.  It is shown that the bifurcation diagram of (\ref{11}) is of pitchfork type;  see Proposition \ref{proposition2}.  However, due to the complexity and difficulty in computations, we only consider the simpler model (\ref{316}) and establish the stability criteria of the bifurcating solutions $(u_k(s,x),v_k(s,x))$.  Our results show that if one of the diffusion rates $D_1$ and $D_2$ is small, the small amplitude solution  $(u_k(s,x),v_k(s,x))$ is unstable.  If the cell motility $D_1$ is large, the bifurcating solution is stable if $D_2<\frac{1}{14}(\frac{L}{k\pi})^2$ and unstable if $D_2>\frac{1}{14}(\frac{L}{k\pi})^2$.  Therefore, we may expect that system (\ref{11}) admits large amplitude solutions in these cases.  Moreover, we establish the existence of nonconstant positive solutions of (\ref{11}) with large amplitude, which has also been formally presented in \cite{MOW}.  Compared with the models studied by Wang and Xu \cite{WX}, our model with logistic cellular growth does not have the feature that the cell population is preserved.  However the logistic growth prevents the solutions from blowing up into a $\delta$-function.  From the viewpoint of mathematical analysis, the logistic growth term inhibits the application of Sturm oscillation theory to (\ref{12}), which is an essential tool that has been used in \cite{WX} to show the emergence of a $\delta$ function as $\chi/D_1 \rightarrow \infty$.

There are also some interesting questions that have not been considered in our paper.  The existence and the structure of nonconstant positive solution to (\ref{43}) can be an interesting question to probe in the future.  Moreover, the stability of the solutions with patterns is also a very interesting and delicate problem that deserves future attention (e.g. \cite{KWA}).  Our results in Theorem \ref{theorem41} suggest that the strength of chemotaxis and cell motility play very important roles in determining the total cell population, therefore it is mathematically interesting and biologically important to find or characterize optimal (large) chemotaxis and diffusion rate that maximize the total cell population.  Obviously one can also investigate the models over a multi-dimensional domain, even for $\Omega$ with special geometries.

\medskip
% The data information below will be filled by AIMS editorial staff
%Received xxxx 20xx; revised xxxx 20xx.
\medskip


\begin{thebibliography}{10}
\bibitem{BWS} {\sc M.D. Baker, P.M. Wolanin and J.B. Stock}, \emph{Signal transduction in bacterial chemotaxis}, Bioessays, \textbf{28} (2006), 9--22.
\bibitem{B} {\sc P. Biler}, \emph{Local and global solvability of some parabolic system modelling chemotaxis}, Adv. Math. Sci. Appl., \textbf{8} (1998), 715--743.
\bibitem{BEG} {\sc  P. Biler, E. Espejo and I. Guerra}, \emph{Blowup in higher dimensional two species chemotactic systems}, Commun. Pure Appl. Anal,, \textbf{12} (2013), 89--98.
\bibitem{CKWW} {\sc A. Chertock, A. Kurganov, X. Wang and Y. Wu}, \emph{On a chemotaxis model with saturated chemotactic flux}, Kinet. Relat. Models, \textbf{5} (2012), 51--95.
\bibitem{CP} {\sc S. Childress and J.K. Percus}, \emph{Nonlinear aspects of chemotaxis}, Math. Biosci., \textbf{56}, (1983), 217--237.
\bibitem{CEV1} {\sc C. Conca, E. Espejo and K. Vilches}, \emph{Remarks on the blowup and global existence for a two species chemotactic Keller--Segel system in $\mathbb R^2$}, European J. Appl. Math., \textbf{22} (2011), 553--580.
\bibitem{CEV2} {\sc C. Conca, E. Espejo and K. Vilches}, \emph{Sharp Condition for blow-up and global existence in a two species chemotactic Keller--Segel system in $\mathbb R^2$}, European J. Appl. Math., \textbf{24} (2013), 297--313.
\bibitem{CR} {\sc M.G. Crandall and P.H. Rabinowitz}, \emph{Bifurcation from simple eigenvalues}, J. Functional Analysis, \textbf{8} (1971) 321--340.
\bibitem{CR2} {\sc M.G. Crandall and P.H. Rabinowitz}, \emph{Bifurcation, perturbation of simple eigenvalues, and linearized stability}, Arch. Ration. Mech. Anal., \textbf{52} (1973) 161--180.
\bibitem{DW} {\sc D. Dormann and C. Weijer}, \emph{Chemotactic cell movement during Dictyostelium development and gastrulation}, Current Opinion in Genetics Development \textbf{16} (2006), 367--373.
\bibitem{EIM} {\sc S.I. Ei, H. Izuhara and M. Mimura}, \emph{Spatio-temporal oscillations in the Keller--Segel system with logistic growth}, Phys. D, \textbf{277}, 2014, 1--21.
\bibitem{HHS}{\sc  M. Henry, D. Hilhorst and R. Schatzle}, \emph{Convergence to a viscosity solution for an advection-reaction-diffusion equation arising from a chemotaxis-growth model}, Hiroshima Math. J., \textbf{29} (1999), 591--630.
\bibitem{HV}  {\sc M.A. Herrero and J.J.L. Velazquez}, \emph{Chemotactic collapse for the Keller--Segel model},  J. Math. Biol., \textbf{35} (1996), 583--623.
\bibitem{HP}{\sc T. Hillen and K.J. Painter}, \emph{A user's guidence to PDE models for chemotaxis},  J. Math. Biol., \textbf{58} (2009), 183--217.
\bibitem{Ho}{\sc  D. Horstmann}, \emph{From 1970 until now: the Keller--Segel model in Chemotaxis and its consequences I}, Jahresber DMV \textbf{105} (2003), 103--165.
\bibitem{Ho1}{\sc  D. Horstmann}, \emph{From 1970 until now: the Keller--Segel model in Chemotaxis and its consequences II}, Jahresber DMV \textbf{106} (2003), 51--69.
\bibitem{Ho3} {\sc  D. Horstmann}, \emph{Generalizing the Keller--Segel model: Lyapunov functionals, steady state analysis, and blow–up results for multi–species chemotaxis models in the presence of attraction and repulsion between competitive interacting species}, J. Nonlinear Sci., \textbf{21} (2011), 231--270.
\bibitem{HW} {\sc D. Horstmann and M. Winkler}, \emph{Boundedness vs. blow-up in a chemotaxisi system}, J. Differential Equations, \textbf{215} (2005), 52--107.
\bibitem{JWZ} {\sc L. Jin, Q. Wang and Z. Zhang}, \emph{Pattern formation in Keller--Segel chemotaxis models with logistic growth}, Int. J. Bifurcation Chaos, \textbf{26} (2016), 1650033-1---1650033-15.
\bibitem{Kato}
    \newblock  T. Kato,
    \newblock    "Functional Analysis",
    \newblock Springer Classics in Mathematics, (1996).
\bibitem{KS} {\sc E.F. Keller and L.A. Segel}, \emph{Inition of slime mold aggregation view as an instability}, J. Theor. Biol., \textbf{26} (1970), 399-415.
\bibitem{KS1} {\sc E.F. Keller and L.A. Segel}, \emph{Model for Chemotaxis}, J. Theor. Biol., \textbf{30}, (1971), 225-234.
\bibitem{KS2} {\sc E.F. Keller and L.A. Segel}, \emph{Traveling bands of chemotactic bacteria: A Theretical Analysis}, J. Theor. Biol. \textbf{30} (1971), 235-248.
\bibitem{KWA} {\sc T. Kolokolnikov, J. Wei and A. Alcolado}, \emph{Basic mechanisms driving complex spike dynamics in a chemotaxis model with logistic growth}, SIAM J. Appl. Math., \textbf{74} (2014), 1375--1396.
\bibitem{KOST}{\sc  K. Kuto, K. Osaki, T. Sakurai and T. Tsujikawa}, \emph{Spatial pattern formation in a chemotaxis-diffusion-growth model}, Phys. D, \textbf{241} (2012), 1629--1639.
\bibitem{KT}{\sc  K. Kuto and T. Tsujikawa}, \emph{Limiting structure of steady-states to the Lotka-Volterra competition model with large diffusion and advection}, J. Differential Equations, \textbf{258} (2015), 1801--1858.
\bibitem{LNT} {\sc C.-S. Lin, W.-M. Ni and I. Takagi}, \emph{Large amplitute stationary solutions to a chemotaxis system}, J. Differential Equations, \textbf{72} (1988), 1-27.
\bibitem{Lou1} {\sc Y. Lou}, \emph{On the effects of migration and spatial heterogeneity on single and multiple species}, J. Differential Equations, \textbf{223} (2006), 400--426.
\bibitem{Lou2} {\sc Y. Lou}, \emph{Some challenging mathematical problems in evolution of dispersal and population dynamics}, Tutorials in mathematical biosciences. IV, 171–205, Lecture Notes in Math., 1922, Springer, Berlin, 2008.
\bibitem{MOW}{\sc M. Ma, C. Ou and Z.-A. Wang}, \emph{Stationary solutions of a volume filling chemotaxis model with logistic growth}, SIAM J. Appl. Math., \textbf{72} (2012), 740--766.
\bibitem{MT}{\sc M. Mimura and T. Tsujikawab}, \emph{Aggregating pattern dynamics in a chemotaxis model including growth}, Phys. A, \textbf{230} (1996), 499--543.
\bibitem{NO} {\sc E. Nakaguchi and K. Osaki}, \emph{Global existence of solutions to a parabolic--parabolic system for chemotaxis with weak degradation}, Nonlinear Anal., \textbf{74} (2011), 286--297.
\bibitem{N} {\sc V. Nanjundiah}, \emph{Chemotaxis, signal relaying and aggregation morphology}, J. Theor. Biol., \textbf{42} 1973, 63--105.
\bibitem{N2}{\sc W.-M. Ni}, \emph{Diffusion, cross-diffusion, and their spike layer steady states}, Notices Amer. Math. Soc., \textbf{15} (1998), 9--18.
\bibitem{NT} {\sc W.-M. Ni and I. Takagi}, \emph{On the shape of least enery solutions to a semilinear Neumann problem}, Comm. Pure Appl. Math., \textbf{44} (1991), 819--851.
\bibitem{NT2} {\sc W.-M. Ni and I. Takagi}, \emph{Location of the peaks of least energy solutions to a semilinear Neumann problem}, Duke Math. J., \textbf{72} (1993), 247--281.
\bibitem{OTYM} {\sc K. Osaki and T. Tsujikawa, A.Yagi and M. Mimura}, \emph{Exponential attractor for a chemotaxis-growth system of equations}, Nonlinear Anal., \textbf{51} (2002), 119--144.
\bibitem{OY} {\sc K. Osaki and A. Yagi}, \emph{Finite dimensional attractor for one-dimensional Keller--Segel equations}, Funkcial. Ekvac., \textbf{44} (2001), 441--469.
\bibitem{PH} {\sc K.J. Painter and T. Hillen}, \emph{Spatio-temporal chaos in a chemotaxis model}, Phys. D, \textbf{240}, 2011, 363--375.
\bibitem{PR} {\sc J. Pejsachowicz and P.J. Rabier}, \emph{Degree theory for $C^1$ Fredholm mappings of index 0},  J. Anal. Math., \textbf{76} (1998), 289--319.
\bibitem{SW} {\sc J. Shi and X. Wang}, \emph{On global bifurcation for quasilinear elliptic systems on bounded domains}, J. Differential Equations, \textbf{246} (2009), 2788-2812.
\bibitem{Sim}  {\sc G. Simonett}, \emph{Center manifolds for quasilinear reaction--diffusion systems}, Differential Integral Equations, \textbf{8} (1995), 753--796.
\bibitem{TW} {\sc J.I. Tello and M, Winkler}, \emph{A chemotaxis system with logistic source}, Comm. Partial Differential Equations, \textbf{32} (2007), 849--877.
\bibitem{TW1} {\sc J. I. Tello and M. Winkler}, \emph{Stabilization in a two-species chemotaxis system with a logistic source}, Nonlinearity, \textbf{25} (2012), 1413--1425.
\bibitem{TW2} {\sc J. I. Tello and M. Winkler}, \emph{Competitive exclusion in a two-species chemotaxis model}, J. Math. Biol., \textbf{68} (2014), 1607--1626.
\bibitem{TKMI} {\sc T. Tsujikawa, K. Kuto, Y. Miyamoto and H. Izuhara}, \emph{Stationary solutions for some shadow system of the Keller-Segel model with logistic growth}, Discrete Contin. Dyn. Syst. Ser. S, \textbf{8} (2015), 1023--1034.
\bibitem{Wq1} {\sc Q. Wang}, \emph{Boundary spikes of a Keller--Segel chemotaxis system with saturated logarithmic sensitivity}, Discrete Contin. Dyn. Syst-Series B, \textbf{20} (2015), 1231--1250.
\bibitem{Wq2} {\sc Q. Wang}, \emph{Global solutions of a Keller Segel system with saturated logarithmic sensitivity function}, Commun. Pure Appl. Anal., \textbf{14} (2015), 383--396.
\bibitem{WGY} {\sc Q. Wang, C. Gai and J. Yan}, \emph{Qualitative analysis of a Lotka-Volterra competition system with advection}, Discrete Contin. Dyn. Syst., \textbf{35} (2015), 1239--1284.
\bibitem{WYZ} {\sc  Q. Wang, J. Yang and L. Zhang}, \emph{Time periodic and stable patterns of a two–competing–species Keller--Segel chemotaxis model: effect of cellular growth}, preprint, arXiv:1505.06463
\bibitem{WZYH} {\sc  Q. Wang, L. Zhang, J. Yang and J. Hu}, \emph{Global existence and steady states of a two competing species Keller--Segel chemotaxis model}, Kinet. Relat. Models, \textbf{8} (2015), 777--807.
\bibitem{W} {\sc X. Wang}, \emph{Qualitative behavior of solutions of chemotactic diffusion systems: effects of motility and chemotaxis and dynamics.},  SIAM J. Math. Anal., \textbf{31} (2000), 535--560.
\bibitem{WW} {\sc X. Wang and Y. Wu}, \emph{Qualitative analysis on a chemotactic diffusion model for two species competing for a limited resource}, Quart. Appl. Math., \textbf{60} (2002), 505--531.
\bibitem{WX} {\sc X. Wang and Q. Xu}, \emph{Spiky and transition layer steady states of chemotaxis systems via global bifurcation and Helly's compactness theorem},  J. Math. Biol., \textbf{66} (2012), 1241--1266.
\bibitem{Wz} {\sc Z.-A. Wang}, \emph{Mathematics of traveling waves in chemotaxis}, Discrete Contin. Dyn. Syst-Series B., \textbf{18} (2013), 601--641.
\bibitem{Wk} {\sc M. Winkler}, \emph{Boundedness in the higher-dimensional parabolic parabolic chemotaxis system with logistic source}, Comm. Partial Differential Equations, \textbf{35} (2010), 1516--1537.
\bibitem{Wk2} {\sc M. Winkler}, \emph{Blow-up in a higher-dimensional chemotaxis system despite logistic growth restriction}, J. Math. Anal. Appl., \textbf{384} (2011), 261--272.









\end{thebibliography}
\end{document}